\documentclass[twoside,10pt,a4paper]{article}
\usepackage[colorlinks]{hyperref}                        
\usepackage{amsmath,latexsym,amsfonts,indentfirst,amsthm,amsxtra,amssymb,bm}
\usepackage[numbers,sort&compress]{natbib}
\usepackage[perpage,symbol*]{footmisc}

\numberwithin{equation}{section}

 \newtheorem{lemma}{Lemma}[section]
 
 \newtheorem{theorem}{Theorem}
 \newtheorem{corollary}[lemma]{Corollary}

 \theoremstyle{remark}
 \newtheorem{remark}{Remark}[section]

\newcommand{\lrn}{\left|\!\left|\!\left|}
\newcommand{\rrn}{\right|\!\right|\!\right|}

\begin{document}

\title{\sc One-Dimensional Compressible Heat-Conducting Gas with
Temperature-Dependent Viscosity}
\author{{\sc Tao Wang}\footnote{School of Mathematics and Statistics,
Wuhan University, Wuhan 430072, China (tao.wang@whu.edu.cn).
The research of this author was supported in part by the Fundamental
Research Funds for the Central Universities
and Project funded by China Postdoctoral Science Foundation.}
\quad and\quad
{\textsc{Huijiang Zhao}}\footnote{Computational Science Hubei Key Laboratory, School of Mathematics and Statistics,
Wuhan University, Wuhan 430072, China (hhjjzhao@whu.edu.cn).
The research of this author was supported in part by grants from the National Natural Science Foundation of China under contracts 10925103, 11271160, and 11261160485, respectively.}}

\date{\today}

\maketitle
\begin{abstract}
We consider the one-dimensional compressible Navier--Stokes system for a viscous and heat-conducting ideal polytropic gas 
when the viscosity $\mu$ and the heat conductivity $\kappa$ depend on the specific volume $v$ and the temperature $\theta$ 
and are both proportional to $h(v)\theta^{\alpha}$ for certain non-degenerate smooth function $h$.
We prove the existence and uniqueness of a global-in-time non-vacuum solution to its Cauchy problem
under certain assumptions on the parameter $\alpha$ and initial data, which imply that 
the initial data can be large if $|\alpha|$ is sufficiently small. 
Our result appears to be the first global existence result for general adiabatic exponent and large initial data
when the viscosity coefficient depends on both the density
and the temperature.

  \bigbreak
  {\bf Keywords:} Compressible Navier--Stokes system;  Temperature-dependent viscosity;
  Ideal polytropic gas;
  Global solutions;
  Large initial data;
  General adiabatic exponent
  \bigbreak
  {\bf Mathematics Subject Classification:}
  35Q35 (35B40, 76N10)
\end{abstract}


	\section{Introduction}
	The one-dimensional motion of a compressible viscous
	and heat-conducting fluid
	can be formulated in the Lagrangian coordinates as
	\begin{equation} \label{NS_L}
		\left\{
		\begin{aligned}
			v_t-u_x&=0,\\
			u_t+P_x&=\left[\frac{\mu u_x}v\right]_x,\\[-0.5mm]
			\left[e+\frac{u^2}{2}\right]_t+(uP)_x&=
			\left[\frac{\kappa\theta_x}v+\frac{\mu uu_x}{v}\right]_x.
		\end{aligned}
		\right.
	\end{equation}
	Here $t>0$ is the time variable, $x\in \mathbb{R}$ is the Lagrangian spatial variable, and the primary dependent variables are the specific volume $v$, the fluid velocity $u$, and the temperature $\theta$.
	The pressure $P$, the specific internal energy $e$, and the transport coefficients $\mu$ (viscosity)
	and $\kappa$ (heat conductivity) are prescribed through constitutive relations as functions of
	the specific volume $v$ and the temperature $\theta$. The thermodynamic variables $v$, $P$, $e$, and $\theta$  are related through  Gibbs equation $\mathrm{d}e=\theta \mathrm{d}s-P\mathrm{d}v$
	with $s$ being the specific entropy.
	
	This paper concerns the construction of globally smooth non-vacuum solutions to the Cauchy problem
	of \eqref{NS_L} for an ideal polytropic gas, which is identified by the constitutive relations
	\begin{equation}\label{ideal}
		P=\frac{\theta}{v}=v^{-\gamma}\exp\left(\frac{s}{c_v}\right),
		\quad e=c_v\theta,
	\end{equation}
	with prescribed initial data
	\begin{equation}\label{initial}
		(v(t,x), u(t,x),\theta(t,x))|_{t=0}
		=(v_0(x), u_0(x), \theta_0(x)) \quad {\rm for }\ x\in\mathbb{R}.
	\end{equation}
	Here $c_v=1/(\gamma-1)$ is the specific heat at constant volume
	with $\gamma>1$ being the adiabatic exponent and
	some gas constants involved have been normalized to be unity
	without loss of generality. It is assumed that the initial data $(v_0, u_0, \theta_0)$ satisfy the far-field condition
	\begin{equation}\label{far}
		\lim_{x\to\pm\infty}(v_0(x), u_0(x), \theta_0(x))=(1,0,1).
	\end{equation}

	We are interested in the case when the transport coefficients $\mu$ and $\kappa$,
	especially the viscosity $\mu$, depend on both the specific volume $v$ and the temperature $\theta$.
	Recall that the study on such a dependence is motivated by the following three observations:
	\begin{list}{\setlength{\parsep}{\parskip}
			\setlength{\itemsep}{0.1em}
			\setlength{\labelwidth}{1.5em}
			\setlength{\labelsep}{0.4em}
			\setlength{\leftmargin}{2em}
			\setlength{\topsep}{1mm}
		}
		\item[(i)] for certain class of
		solid-like materials considered
		in  \cite{Da82MR653464} and \cite{DH82MR661710},
		both the viscosity coefficient $\mu$
		and the heat conductivity coefficient $\kappa$ may depend on
		the density and/or temperature;
		\item[(ii)]
		experimental results in  \cite{ZR67physics}
		show that the transport coefficients $\mu$ and $\kappa$ vary in terms of
		temperature and density
		for gases at very high
		temperature and density;
		\item[(iii)] if the compressible
		Navier--Stokes equations \eqref{NS_L}
		are
		derived from the
		Boltzmann equation with slab symmetry
		for the monatomic gas
		by using the Chapman--Enskog expansion,
		the constitutive relations between thermodynamic variables
		satisfy \eqref{ideal} and the transport coefficients
		$\mu$ and $\kappa$ depend only on the temperature.
		Moreover,
		the functional dependence is the same for both coefficients
		(see  \cite{CC90MR1148892}
		and \cite{VK65}).
		In particular, if the intermolecule potential
		varies as $r^{-a}$ with $r$ being the molecule distance, then
		\begin{equation}\label{transport0}
			\mu=\tilde{\mu}\theta^\alpha,\quad \kappa=\tilde{\kappa}\theta^\alpha,
		\end{equation}
		where $\tilde{\mu}$, $\tilde{\kappa}$,
		and $\alpha=\tfrac{a+4}{2a}>\tfrac 12$  are positive constants.
	\end{list}
	
	
	The crucial step to construct the global solutions of
	the compressible Navier--Stokes equations
	\eqref{NS_L} with large initial data
	is to obtain the positive upper and lower bounds
	of the specific volume $v$ and the temperature $\theta$,
	which has been shown in  \cite{KO82MR694940}
	for small and sufficiently smooth data.
	When the viscosity and the heat conductivity coefficients
	are positive constants, Kazhikhov et al. \cite{AKM90MR1035212,K82MR651877,KS77MR0468593}
	succeeded in deriving a representation for specific volume $v$
	by employing the special structure of ideal polytropic gases
	\eqref{NS_L}--\eqref{ideal}. By mean of the representation
	for $v$ and the maximum principle, the positive upper and lower bounds
	of $v$ and $\theta$ as well as the existence and uniqueness of
	globally smooth solutions have been obtained
	in  \cite{AKM90MR1035212,K82MR651877} and \cite{KS77MR0468593}
	for \eqref{NS_L}--\eqref{ideal} with arbitrarily large initial data.
	See also 
	\cite{AZ92MR1187869,J98MR1748226,J99MR1671920,J02MR1912419,KN81MR637519} and 
	\cite{ZA97MR1474911} for related studies.
	In all of these works no vacuum nor concentration
	of mass occur in a finite time.

	We note that this argument can be applied to the case
	when the viscosity $\mu$ is a constant and
	the heat conductivity $\kappa$ is some function of
	temperature $\theta$
	(see  \cite{JK10MR2644363,PZ15MR3291375} and
	\cite{W16MR3461630}).
	But this methodology seems not valid if the viscosity $\mu$
	is a non-constant function of  $v$ and  $\theta$.
	For the case when the viscosity $\mu$ is a function of
	the specific volume $v$ alone, as observed by
	Kanel$'$ \cite{Ka68MR0227619} for the
	isentropic flow, the identity
	\begin{equation}\label{id_Kanel}
		\left[\frac{\mu(v)v_x}{v}\right]_t=u_t+P_x
	\end{equation}
	holds even for general gases.
	By employing this identity, one can
	deduce global solvability results on the compressible
	Navier--Stokes equations \eqref{NS_L} with large data
	for certain types of
	density-dependent viscosity, and density and temperature
	dependent heat conductivity.
	See  \cite{CHT00MR1789926,Da82MR653464,DH82MR661710,Ka85MR791841,TYZZMR3032988},
	and references therein
	for some representative works in this direction.
	
	When the viscosity $\mu$ depends on the temperature $\theta$
	and the specific volume $v$,
	the identity corresponding to \eqref{id_Kanel} becomes
	\begin{equation}\label{id_TDV}
		\left[\frac{\mu (v,\theta)v_x}{v}\right]_t
		=u_t+P_x+
		\frac{\mu_{\theta}(v,\theta)  }{v}\left(\theta_tv_x-u_x\theta_x\right)
	\end{equation}
	with $\mu_{\theta}(v,\theta):={\partial \mu(v,\theta)}/{\partial \theta}$. The temperature dependence
	of the viscosity $\mu$ has a strong influence on the solution and leads to
	difficulty in mathematical analysis for global solvability with large data.
	As pointed out in \cite{JK10MR2644363}, such a dependence has turned
	out to be especially problematic and challenging. One of the main difficulties in analysis
	arises from the last term in \eqref{id_TDV}, which is a highly nonlinear term.
	
	A possible way to go on is to use some ``smallness mechanism" induced
	by the structure of the equations \eqref{NS_L} to control the last term
	in \eqref{id_TDV} suitably. A recent progress along this way is
	a Nishida--Smoller type global solvability result with large data
	obtained in  \cite{LYZZMR3225502} for the Cauchy problem
	\eqref{NS_L}--\eqref{far} when the
	viscosity
	$\mu$ and the
	heat conductivity $\kappa$ are both functions of the temperature.
	The main observation in  \cite{LYZZMR3225502} is that for
	ideal polytropic gases \eqref{NS_L}--\eqref{ideal},
	the temperature $\theta$ satisfies
	\begin{equation*}
		\theta=v^{1-\gamma}\mathrm{e}^{(\gamma-1)s}
		\quad {\rm and}  \quad 
		\frac{\theta_t}{\gamma-1}+\frac{\theta u_x}{v}=
		\frac{\mu u^2_x}{v}+\left[\frac{\kappa\theta_x}{v}\right]_x,
	\end{equation*}
	from which one can deduce that
	$\|(\theta-1,\theta_t ,\theta_x )\|_{L^\infty([0,T]\times\mathbb{R})}$
	can be small under the condition that
	the adiabatic exponent $\gamma$ is close to $1$.
	Thus one can perform the desired energy-type
	a priori estimates as in 
	\cite{DLZ09MR2439413,KN81MR637519,NYZ04MR2083790}
	and \cite{OK83MR692729}
	based on the a priori assumption
	\begin{equation}\label{key2_TYZZ}
		\tfrac 12\leq\theta(t,x)\leq 2\quad {\rm for\ all}\
		(t,x)\in[0,T]\times\mathbb{R}.
	\end{equation}
	It is to close the a priori assumption \eqref{key2_TYZZ}
	on $\theta(t,x)$ that one needs to impose that the initial
	data satisfies a Nishida--Smoller type condition, that is,
	\begin{equation*}
		(\gamma-1)\times C\big(\|(v_0-1,u_0,s_0-1)\|_{H^3(\mathbb{R})},
		\inf\limits_{x\in{\mathbb{R}}}v_0(x)\big)\leq 1
	\end{equation*}
	for some  $(\gamma-1)$-independent smooth function $C$.

	The result obtained in  \cite{LYZZMR3225502}
	shows that $\|(v_0-1,u_0,s_0-1)\|_{H^3(\mathbb{R})}$
	can be large. However,
	the oscillation
	of the temperature, $\|\theta-1\|_{L^\infty([0,T]\times\mathbb{R})}$,
	is not arbitrarily large but has to be small.
	Thus a natural question is:
	{\it Whether can we obtain a global solvability result
		for the Cauchy problem \eqref{NS_L}--\eqref{far}
		with large initial data and general adiabatic exponent $\gamma$
		for a class of temperature and  density dependent
		viscosity coefficient $\mu$ or not?}
	
	The main goal of this paper is devoted to
	the above problem and our motivation is essentially
	the same as that of  \cite{LYZZMR3225502} mentioned above,
	that is {\it to use some ``smallness mechanism"
		induced by the structure of the equations \eqref{NS_L}
		to control the last term  in \eqref{id_TDV} suitably.}
	We note that
	$\gamma-1$ cannot be assumed to be small
	for the case with general adiabatic exponent.
	In this paper, we will
	try to use the smallness of $|\mu_\theta(v,\theta)|$
	to control the possible growth of the solutions
	of the Cauchy problem \eqref{NS_L}--\eqref{far}
	caused by the last term in \eqref{id_TDV}.
	Motivated by such an idea,
	we assume throughout the rest of this paper
	that the viscosity $\mu$ and the heat conductivity
	$\kappa$ are smooth functions of the temperature
	$\theta$ and the specific volume $v$, which are given by
	\begin{equation}\label{transport}
		\mu=\tilde{\mu}h(v)\theta^{\alpha},\quad
		\kappa=\tilde{\kappa}h(v)\theta^{\alpha},
	\end{equation}
	where $\tilde{\mu}$ and $\tilde{\kappa}$ are positive constants,
	and there exist positive constants $C$, $\ell_1$, and $\ell_2$
	such that
	\begin{equation}\label{h}
		C h(v)\geq v^{\ell_1}+v^{-\ell_2},
		\quad h'(v)^2v\leq C h(v)^3 \quad {\rm for\ all}\ v\in(0,\infty).
	\end{equation}
	We expect to obtain a global solvability result to
	the Cauchy problem  \eqref{NS_L}--\eqref{far}
	with large data and transport coefficients
	\eqref{transport}--\eqref{h}
	for general adiabatic exponent
	$\gamma$ provided that $|\alpha|$ is sufficiently small.
	
	The very reason why we choose
	$\mu$ and $\kappa$
	as in \eqref{transport}
	is that the transport coefficients
	\eqref{transport}--\eqref{h} with
	$\ell_1=\ell_2=0$ can include \eqref{transport0}
	as a special example.
	Moreover,
	the special form \eqref{transport} of
	the viscosity $\mu$ with $h(v)$ satisfying \eqref{h}
	is essential
	in our argument and the role is two-fold:
	\begin{list}{\setlength{\parsep}{\parskip}
			\setlength{\itemsep}{0.1em}
			\setlength{\labelwidth}{1.5em}
			\setlength{\labelsep}{0.4em}
			\setlength{\leftmargin}{2em}
			\setlength{\topsep}{1mm}
		}
		\item [(i)] Firstly,
		we will employ the smallness of resulting factor $|\alpha|$
		to control the last term in \eqref{id_TDV};
		\item [(ii)] Secondly, the assumption \eqref{h} imposed on $h(v)$
		will be used to yield some estimates
		on the lower and upper bounds for the specific volume $v(t,x)$
		in terms of $\|\theta\|_{L^\infty([0,T]\times\mathbb{R})}$.
	\end{list}
	As for the heat conductivity $\kappa$,
	the choice as in \eqref{transport} is not so crucial and
	can be replaced by some more general function of $v$ and $\theta$
	which satisfies certain conditions in terms of
	the parameters $\ell_1$, $\ell_2$, and $\alpha$.
	Such a generalization is straightforward
	and hence we will focus on the case when $\kappa$ is given by
	\eqref{transport}--\eqref{h} for simplicity of presentation.

	We introduce
	\begin{equation}
		\label{H}
		H(w):=
		\sup_{w\leq \sigma\leq w^{-1}}\left|\left(h(\sigma), h'(\sigma), h''(\sigma), h'''(\sigma)\right)\right|
		\quad \textrm{for } w>0,
	\end{equation}
	and state
	our main result as follows.
	\begin{theorem}
		\label{thm} Assume that the viscosity $\mu$ and the heat conductivity $\kappa$
		satisfy \eqref{transport}--\eqref{h} for some $\ell_1\geq 1$
		and $\ell_2\geq1$.
		Let  the initial data  $(v_0, u_0, \theta_0)$ satisfy that
		\begin{gather}\label{thm1a}
			(v_0-1,u_0,\theta_0-1)\in H^3(\mathbb{R}),\quad 
			\|(v_0 -1,u_0 ,\theta_0 -1)\|_{H^3(\mathbb{R})}\leq \Pi_0,\\
			\label{thm1b}
			V_0\leq v_0(x)\leq V_0^{-1},\quad
			\theta_0(x)\geq V_0\quad
			\textrm{for\ all}\quad x\in\mathbb{R},
		\end{gather}
		where $\Pi_0$ and $V_0$ are positive constants.
		Then there exists $\epsilon_0>0$,
		which depends only on
		$\Pi_0$, $V_0$, and
		$H(C_0)$
		with positive constant $C_0$
		depending only on
		$\Pi_0$, $V_0$,
		and $H(V_0)$,
		such that
		the Cauchy problem
		\eqref{NS_L}--\eqref{far}
		with $|\alpha|\leq \epsilon_0$
		admits a unique
		solution  $(v(t,x), u(t,x), \theta(t,x))$ satisfying
		\begin{gather} \label{thm2a}
			(v -1,u,\theta -1)\in C([0,\infty),H^3(\mathbb{R})),\\
			\label{thm2b}
			v_x \in L^2(0,\infty;H^2(\mathbb{R})),
			\quad
			(u_x,\theta_x )\in L^2(0,\infty;H^3(\mathbb{R})),
		\end{gather}
		and
		\begin{equation}\label{thm3}
			\inf_{(t,x)\in[0,\infty)\times\mathbb{R}}
			\left\{v(t,x),\theta(t,x)\right\}>0,\quad
			\sup_{(t,x)\in[0,\infty)\times\mathbb{R}}
			\left\{v(t,x),\theta(t,x)\right\}<+\infty.
		\end{equation}
		Furthermore, the solution $(v, u, \theta)$
		converges to $(1,0,1)$ uniformly as time
		tends to infinity:
		\begin{equation}\label{thm4}
			\lim_{t\to \infty}\sup_{x\in\mathbb{R}}
			\big|(v(t,x)-1,u(t,x),\theta(t,x)-1)\big|=0.
		\end{equation}
	\end{theorem}
	\begin{remark}
		We deduce from \eqref{thm1a}--\eqref{thm1b} and \eqref{thm3} that no
		vacuum will be developed if
		the initial data
		do not contain a vacuum.
		It follows from \eqref{thm2a}--\eqref{thm2b} and Sobolev's imbedding theorem
		that the unique solution constructed in Theorem \ref{thm}
		is a globally smooth non-vacuum solution with large initial data.
		Moreover, this result in Lagrangian coordinates
		can easily be converted to equivalent statement
		for the corresponding problem in Eulerian coordinates.
	\end{remark}
	\begin{remark} As far as we are aware,
		for ideal polytropic gases with general adiabatic exponent $\gamma$,
		Theorem \ref{thm} is the first result on
		the global well-posedness of smooth non-vacuum solutions
		to the compressible Navier--Stokes equations \eqref{NS_L}--\eqref{ideal}
		with temperature-dependent viscosity and large initial data.
	\end{remark}
	\begin{remark}
		The assumption we imposed on the parameters $\ell_1$ and
		$\ell_2$
		in Theorem \ref{thm} is just for illustrating our main idea
		to deduce the desired result and is far from being optimal.
		In fact, our approach can be applied
		to prove a similar global solvability result
		when the parameters $\ell_1$ and $\ell_2$ satisfy
		\begin{equation*} 
			\ell_1>0,\quad \ell_2>0\quad {\rm and}\quad
			\tfrac{18}{1+2\ell_1}\max\left\{0,1-\ell_1\right\}
			+\tfrac{19}{2\ell_2}\max\left\{0,1-\ell_2\right\}
			<4.
		\end{equation*}
		Unfortunately, our result cannot cover the model satisfying
		\eqref{transport0} since the parameters $\ell_1$ and
		$\ell_2$ are assumed to be positive.
		The extension of our result to the case with
		$\ell_1=\ell_2=0$ is an open
		problem for future research.
		
	\end{remark}
	
	\begin{remark}
		The existence of global strong solutions to the one-dimensional compressible Navier--Stokes equations for isentropic flows has been established in  \cite{MV08MR2368905}  with  the viscosity $\mu$ given by
		\eqref{transport}--\eqref{h}  for $\ell_1=0$, $0\leq \ell_2<\tfrac{1}{2}$, and $\alpha=0$,	and also  in  \cite{Haspot1411.5503} for the shallow water system, where the viscosity $\mu$ satisfies \eqref{transport} with $h(v)=v^{-1}$ and $\alpha=0$.
		Note that
		our derivation of the uniform bounds on $v(t,x)$ and $\theta(t,x)$
		relies heavily on the assumption that the initial data is sufficiently smooth.
		It is an interesting and difficult problem
		to extend the results in   \cite{Haspot1411.5503} and \cite{MV08MR2368905}
		to the non-isentropic case with transport coefficients satisfying \eqref{transport} for nonzero $\alpha$.
	\end{remark}
	Now we outline the main ideas to deduce our main result
	Theorem \ref{thm}. As pointed out before, the key point for
	the global solvability result with large data is to deduce the
	desired positive lower and upper bounds on the specific volume
	$v(t,x)$ and the temperature $\theta(t,x)$ uniformly in space $x$
	as in  \cite{K82MR651877,KS77MR0468593}
	and \cite{TYZZMR3032988}.
	Since we are trying to use the smallness of $|\alpha|$ to
	control the possible growth of the solutions caused by the last term
	in \eqref{id_TDV}, the amplitude of $|\alpha|$ should be determined
	by  pointwise bounds for the specific volume $v(t,x)$
	and the temperature $\theta(t,x)$.
	The main point in our analysis is to determine the positive
	parameter $\epsilon_0$
	(namely, the upper bound of $|\alpha|$)
	in Theorem \ref{thm}
	in terms of the initial data,
	such that the whole analysis can be carried out for
	$|\alpha|\leq \epsilon_0$.
	To guarantee the existence of such an $\epsilon_0$
	(i.e. to insure that the parameter
	$\alpha$ does not vanish) when we extend the
	local solutions step by step to the global ones,
	we have to obtain the
	lower and upper bounds for $v(t,x)$ and $\theta(t,x)$
	uniformly in time $t$ and space $x$.
	It is worth noting that,
	even for the Cauchy problem \eqref{NS_L}--\eqref{far} with constant
	transport coefficients, such uniform bounds
	on $\theta(t,x)$ are obtained
	only very recently by  Li and Liang\cite{LL},
	although the corresponding global solvability
	result was addressed by Kazhikhov\cite{K82MR651877}
	a long time ago.
	The starting points of the argument in  \cite{LL}
	are the following:
	\begin{list}{\setlength{\parsep}{\parskip}
			\setlength{\itemsep}{0.1em}
			\setlength{\labelwidth}{1.5em}
			\setlength{\labelsep}{0.4em}
			\setlength{\leftmargin}{2em}
			\setlength{\topsep}{1mm}
		}
		\item [(i)] the global  existence result obtained
		in  \cite{K82MR651877};
		\item [(ii)] the uniform positive
		lower and upper bounds on $v(t,x)$ obtained in
		 \cite{J99MR1671920} and \cite{J02MR1912419}  by using a decent	localized version of the expression for $v(t,x)$.
	\end{list}
	Based on these two points, Li and Liang further
	deduce the uniform positive lower and upper bounds
	on the temperature $\theta(t,x)$ in  \cite{LL}
	through a time-asymptotically nonlinear stability analysis.
	However, the approach in  \cite{J99MR1671920} and \cite{J02MR1912419}
	cannot be applied to the case
	when the viscosity $\mu$ is a  non-constant
	function of $v$ and $\theta$.
	To overcome such a difficulty,
	we employ the argument developed by Kanel$'$
	\cite{Ka68MR0227619,LYZZMR3225502} to prove that
	the specific volume $v(t,x)$ can be bounded
	in terms of the upper bound of the temperature $\theta(t,x)$.
	Then we combine the local-in-time lower bound on the
	temperature $\theta(t,x)$ induced by the maximum principle
	and a well-designed continuation argument
	to obtain the positive lower and upper bounds of
	the temperature $\theta(t,x)$ uniformly in time and space
	as well as the global existence of smooth solutions.
	Such a continuation argument is of some interest itself
	and can be used to study some other
	problems, such as nonlinear stability
	of the non-degenerate stationary solutions to
	the outflow problem of the compressible Navier--Stokes equations
	\eqref{NS_L}--\eqref{ideal} with large initial perturbation
	and general adiabatic exponent $\gamma$ in  \cite{WWZ15}.

	Before concluding this section, let us point out that
	our result shows that
	no vacuum, mass or heat concentration will be developed in any
	finite time, although the motion of the flow has large oscillations.
	For the corresponding results on the compressible
	Navier--Stokes equations
	with large data and vacuum,
	we refer to  \cite{BD07MR2297248,Fe07MR2323704,LJX08MR2410901,LXY98MR1485360,YZ02CMPMR1936794},
	and the references therein.

	The layout of the rest of this manuscript is organized as follows.
	In subsection \ref{sec_v},
	we deduce the estimate for $\left\|\frac{\mu v_x}{v}(t)\right\|$
	under some a priori assumptions as in Lemma \ref{L_v1},
	and by applying the argument developed by Kanel$'$,
	we prove in Lemma \ref{L_v2} that the bounds of the specific volume $v(t,x)$ can be controlled
	in terms of the upper bound of the temperature $\theta(t,x)$.
	In subsection \ref{sec_theta}, we estimate the $H^1(\mathbb{R})$-norm
	of the temperature $\theta(t,x)$ and obtain the upper and lower
	bound on the temperature $\theta(t,x)$.
	The estimates on second-order and third-order
	derivatives of the solution $(v(t,x), u(t,x), \theta(t,x))$ will be deduced in
	subsections \ref{sec_2} and \ref{sec_3}, respectively.
	Finally, in Sec. \ref{sec_proof},
	by combining the  a priori estimates
	and a well-designed continuation argument,
	we derive the positive lower and upper bounds of
	the temperature $\theta(t,x)$ and the specific volume $v(t,x)$ uniformly in time and space
	and extend the local solution step by step to the global one.
	
	\vspace*{2mm}
	\noindent\emph{Notations.}
	Throughout this paper,
	$L^q(\mathbb{R} )$ $(1\leq q\leq \infty)$ stands for
	the usual Lebesgue space on $\mathbb{R} $ with norm $\|{\cdot}\|_{L^q}$
	and $H^k(\mathbb{R})$  $(k\in \mathbb{N})$
	the usual Sobolev space in the $L^2$ sense
	with norm $\|\cdot\|_k$.
	We introduce $\|\cdot\|=\|\cdot\|_{L^2(\mathbb{R})}$ for
	notational simplicity.
	We denote by $C(I; H^p)$ the space of  continuous
	functions on the interval $I$ with values in
	$H^p(\mathbb{R})$ and $L^2(I; H^p)$
	the space of $L^2$-functions
	on $I$ with values in $H^p(\mathbb{R})$.
	We introduce
	$A\lesssim B$ (or $B\gtrsim A$)
	if $A\leq C B$ holds uniformly for some constant $C$
	depending solely on $\Pi_0$,
	$V_0$, and
	$H(V_0)$, where
	$\Pi_0$,
	$V_0$, and
	$H$ are given by \eqref{H} and \eqref{thm1a}--\eqref{thm1b}.
	
	\section{A Priori Estimates}
	We define, for constants $N$, $m_i$, $s$, and $t$ ($i=1,2$, $t\geq s$),
	the set
	\begin{equation*}
		\begin{split}
			&X(s,t ;m_1,m_2,N):=\big\{(v, u,\theta):
			(v-1,u,\theta-1)\in C([s,t];H^3),\\
			&\qquad\qquad v_x\in L^2(s,t;H^2),\
			(u_x,\theta_x)\in L^2(s,t;H^3),\\
			&\qquad\qquad 
			\mathcal{E}(s,t)\leq N^2,\ v(\tau,x)\geq m_1,\
			\theta(\tau,x)\geq m_2\ \forall\,  
			(\tau,x)\in[s,t ]\times\mathbb{R}
			\big\},
		\end{split}
	\end{equation*}
	where
	\begin{equation*}
		\mathcal{E}(s,t):=\sup_{\tau\in[s,t]}\|(v-1,u,\theta-1)(\tau)\|_3^2
		+\int_s^t\left[\|v_x(\tau)\|_2^2+\|(u_x,\theta_x)(\tau)\|_3^2\right]
		\mathrm{d}\tau.
	\end{equation*}
	The main purpose of this section
	is to derive certain a priori estimates on
	the solution $(v,u, \theta)\in X(0,T;m_1,m_2,N)$
	to the Cauchy problem
	\eqref{NS_L}--\eqref{far}
	with constitutive relations  \eqref{transport}
	and \eqref{h}
	for
	$T>0$ and
	$0<m_i\leq 1\leq N<+\infty$ $(i=1,2)$.
	It follows from the Sobolev's inequality that
	\begin{equation}\label{apriori1}
		m_1\leq v(t,x)\leq 4 N, \quad
		m_2\leq \theta(t,x)\leq 4 N\quad
		{\rm for\ all}\ (t,x)\in[0,T]\times\mathbb{R}.
	\end{equation}
	
	To make the presentation clearly, we
	divide this section into the following four parts, where
	we use
	$\lrn\cdot\rrn:=\|\cdot\|_{L^\infty([0,T]\times\mathbb{R})}$
	for notational simplicity.
	\subsection{Pointwise bounds on specific volume}\label{sec_v}
	In this part, we will deduce the lower and upper bounds on the specific volume $v(t,x)$ in terms of $\lrn\theta\rrn$.
	To this end, we first have the basic energy estimate.
	\begin{lemma}
		\label{L_bas}
		Assume that the conditions listed
		in Theorem \ref{thm} hold. Then
		\begin{equation}\label{E_basic}
			\sup_{t\in[0,T]}
			\int_{\mathbb{R}}\eta(v,u,\theta)(t,x)\mathrm{d}x
			+\int_0^T\int_{\mathbb{R}}\left[\frac{\mu u_x^2}{v\theta}
			+\frac{\kappa\theta_x^2}{v\theta^2}\right]
			\lesssim 1,
		\end{equation}
		where
		\begin{align}\label{eta}
			\eta(v,u,\theta)&:=\phi(v)+\tfrac12u^2+c_v\phi(\theta),
			\\ \label{phi}
			\phi(z)&:=z-\ln z-1.
		\end{align}
	\end{lemma}
	\begin{proof}
		In light of \eqref{NS_L}, we deduce
		\begin{equation}\label{id_theta}
			c_v{\theta_t}+\frac{\theta u_x}{v}=
			\left[\frac{\kappa\theta_x}{v}\right]_x+\frac{\mu u^2_x}{v}.
		\end{equation}
		Multiplying \eqref{NS_L}$_1$
		(the first equation of \eqref{NS_L}), \eqref{NS_L}$_2$,
		and \eqref{id_theta} by $(1-v^{-1})$, $u$, and $(1-\theta^{-1})$,
		respectively,   we find
		\begin{equation*}
			\eta(v,u,\theta)_t+\frac{\mu u_x^2}{v\theta}
			+\frac{\kappa\theta_x^2}{v\theta^2}
			=\left[\frac{\mu uu_x}{v}+\left(1-\frac{1}{\theta}\right)
			\frac{\kappa\theta_x}{v}
			+\left(1-\frac{\theta}{v}\right)u\right]_x.
		\end{equation*}
		Integrate the above identity over $[0,T]\times\mathbb{R}$
		to have
		\begin{equation} \label{E_basic1}
			\int_{\mathbb{R}}\eta(v,u,\theta)(t,x)\mathrm{d}x
			+\int_0^t\int_{\mathbb{R}}\left[\frac{\mu u_x^2}{v\theta}
			+\frac{\kappa\theta_x^2}{v\theta^2}\right]
			=\int_{\mathbb{R}}\eta(v_0,u_0,\theta_0)(x)\mathrm{d}x.
		\end{equation}
		It follows from the identity
		$
		\phi(z)=\int_0^1\int_0^1\theta_1\phi''(1+\theta_1\theta_2(z-1))
		\mathrm{d}\theta_2\mathrm{d}\theta_1
		(z-1)^2
		$
		that
		\begin{equation} \label{E_phi}
			(z+1)^{-2}(z-1)^2\lesssim\phi(z)\lesssim (z^{-1}+1)^2(z-1)^2.
		\end{equation}
		Applying the last inequality to $\phi(v_0)$ and $\phi(\theta_0)$,
		we obtain
		\begin{equation*}
			\eta(v_0,u_0,\theta_0)(x)\lesssim 1.
		\end{equation*}
		Plug this last inequality into \eqref{E_basic1}
		to derive \eqref{E_basic}.
		The proof of this lemma is completed.
	\end{proof}
	Our analysis will rely on
	the following lemma.
	\begin{lemma}
		\label{L_v1}
		Suppose that the conditions listed
		in Theorem \ref{thm} hold.
		Then there is a constant $0<\epsilon_1\leq 1$,
		depending only on $\Pi_0$,
		$V_0$, and
		$H(V_0)$,
		such that if
		\begin{gather}\label{apriori3}
			m_2^{-|\alpha|}\leq 2,\quad N^{|\alpha|}\leq 2,\quad
			\Xi(m_1,m_2,N)|\alpha|\leq \epsilon_1,
		\end{gather}
		where $$
		\Xi(m_1,m_2,N):=
		\left[m_1^{-1}+m_2^{-1}+N+
		\sup_{m_1\leq \sigma \leq 4N}h(\sigma)+1\right]^{80},$$
		then
		\begin{equation}\label{E_vx}
			\sup_{t\in[0,T]}\left\|\frac{\mu v_x}{v}(t)\right\|^2
			+\int_0^T\int_{\mathbb{R}}\frac{\mu \theta v_x^2}{v^3}
			\lesssim 1+\lrn\theta\rrn.
		\end{equation}
	\end{lemma}
	\begin{proof}
		According to the chain rule, we have
		\begin{equation*}
			\begin{aligned}
				\left(\frac{\mu v_x}{v}\right)_t
				&=\mu\left(\frac{v_t}{v}\right)_x
				+\frac{v_x}{v}(\mu_v v_t+\mu_{\theta}\theta_t)\\
				&=\left(\frac{\mu v_t}{v}\right)_x
				-\frac{v_t}{v}\mu_x
				+\frac{v_x}{v}(\mu_v v_t+\mu_{\theta}\theta_t)\\
				&=\left(\frac{\mu v_t}{v}\right)_x
				+\frac{\mu_{\theta}}{v}(v_x \theta_t-\theta_x v_t),
			\end{aligned}
		\end{equation*}
		which combined with \eqref{NS_L} implies
		\begin{equation}\label{id1}
			\left(\frac{\mu v_x}{v}\right)_t
			=u_t+\left(\frac{\theta}{v}\right)_x
			+\frac{\mu_{\theta}}{v}(v_x \theta_t-\theta_x u_x).
		\end{equation}
		Multiply \eqref{id1} by ${\mu v_x}/{v}$ to deduce
		\begin{equation*}
			\begin{aligned}
				&\left[\frac12\left(\frac{\mu v_x}{v}\right)^2\right]_t
				+\frac{\mu\theta v_x^2}{v^3}+\left(\frac{\mu uu_x}{v}\right)_x
				-\left(\frac{\mu v_x}{v}u\right)_t\\
				&\qquad
				=\frac{\mu u_x^2}{v}
				+\frac{\mu v_x\theta_x}{v^2}
				+\frac{\mu_{\theta}}{v^2}(\mu v_x-uv)(v_x \theta_t-\theta_x u_x),
			\end{aligned}
		\end{equation*}
		and hence
		\begin{equation*}
			\begin{aligned}
				&\left\|\frac{\mu v_x}{v}(t)\right\|^2
				+\int_0^t\int_{\mathbb{R}}\frac{\mu\theta v_x^2}{v^3}
				\lesssim 1+\left|\int_{\mathbb{R}}\frac{\mu v_x}{v}u
				\mathrm{d}x\right|
				+\int_0^t\int_{\mathbb{R}}\frac{\mu u_x^2}{v} \\
				&\qquad
				+\left|\int_0^t\int_{\mathbb{R}}\frac{\mu v_x\theta_x}{v^2}
				\right|
				+\left|\int_0^t\int_{\mathbb{R}}
				\frac{\mu_{\theta}}{v^2}(\mu v_x-uv)(v_x \theta_t-\theta_x u_x)
				\right|.
			\end{aligned}
		\end{equation*}
		We deduce from Cauchy's inequality and \eqref{E_basic} that
		\begin{align} \notag
			&\left\|\frac{\mu v_x}{v}(t)\right\|^2
			+\int_0^t\int_{\mathbb{R}}\frac{\mu\theta v_x^2}{v^3}
			\lesssim 1+\int_0^t\int_{\mathbb{R}}\frac{\mu u_x^2}{v}\\
			&\qquad  \label{E_vx1}
			+\int_0^t\int_{\mathbb{R}}\frac{\mu \theta_x^2}{v\theta}
			+\left|\int_0^t\int_{\mathbb{R}}
			\frac{\mu_{\theta}}{v^2}(\mu v_x-uv)(v_x \theta_t-\theta_x u_x)\right|.
		\end{align}
		We first estimate the last term in \eqref{E_vx1}.
		It follows from \eqref{id_theta} that
		\begin{equation} \label{id_theta1}
			\theta_t=\frac{1}{c_v}\left[\frac{\kappa_v\theta_x}{v}
			-\frac{\kappa\theta_x}{v^2}\right]v_x
			+\frac{1}{c_v}\left[\frac{\kappa_{\theta}\theta_x^2}{v}
			+\frac{\kappa\theta_{xx}}{v}+\frac{\mu u_x^2}{v}
			-\frac{\theta u_x}{v}\right],
		\end{equation}
		which yields
		\begin{equation}\label{id2}
			\begin{aligned}
				(\mu v_x-uv)(v_x \theta_t-\theta_x u_x)
				=uv\theta_x u_x+\mathcal{R}_1v_x+\mathcal{R}_2v_x^2
			\end{aligned}
		\end{equation}
		with
		\begin{align*}
			\mathcal{R}_1&:=-\mu\theta_x u_x
			-\frac{u}{c_v}\left(\kappa_{\theta}\theta_x^2+\kappa\theta_{xx}
			+\mu u_x^2-\theta u_x\right),
			\\ 
			\mathcal{R}_2&:=\mu\theta_t-\frac{u}{c_v}\left[
			\kappa_v\theta_x-\frac{\kappa\theta_x}{v}\right].
		\end{align*}
		Plug \eqref{id2} into \eqref{E_vx1} to obtain
		\begin{align} \notag
			&\left\|\frac{\mu v_x}{v}(t)\right\|^2
			+\int_0^t\int_{\mathbb{R}}\frac{\mu\theta v_x^2}{v^3}
			\lesssim 1+\int_0^t\int_{\mathbb{R}}\frac{\mu u_x^2}{v}
			+\int_0^t\int_{\mathbb{R}}\frac{\mu \theta_x^2}{v\theta} \\
			\label{E_vx2}
			&\qquad
			+\left|\int_0^t\int_{\mathbb{R}}\frac{\mu_{\theta}}{v}u\theta_x u_x \right|
			+\left|\int_0^t\int_{\mathbb{R}}\frac{\mu_{\theta}}{v^2}v_x
			\mathcal{R}_1  \right|
			+\left|\int_0^t\int_{\mathbb{R}}\frac{\mu_{\theta}}{v^2}v_x^2
			\mathcal{R}_2 \right|.
		\end{align}
		Next we estimate the terms on the right-hand side of \eqref{E_vx2}.
		In view of \eqref{E_basic}, we have
		\begin{equation*}
			\left|\int_0^t\int_{\mathbb{R}}\frac{\mu_{\theta}}{v}u\theta_x u_x \right|
			\leq \lrn\frac{\mu_{\theta}\sqrt{\theta}\theta}{\sqrt{\kappa\mu}}u\rrn
			\int_0^t\left\|\frac{\sqrt{\kappa}\theta_x}{\sqrt{v}\theta}\right\|
			\left\|\frac{\sqrt{\mu}u_x}{\sqrt{v\theta}}\right\|
			\lesssim \lrn\frac{\mu_{\theta}\sqrt{\theta}\theta}{\sqrt{\kappa\mu}}u\rrn.
		\end{equation*}
		We deduce from the identity $\mu_{\theta}=\alpha\mu/\theta$
		and \eqref{apriori1} that
		\begin{equation*}
			\lrn\frac{\mu_{\theta}\sqrt{\theta}\theta}{\sqrt{\kappa\mu}}u\rrn
			\lesssim\lrn \alpha\sqrt{\theta}u\rrn
			\lesssim |\alpha|  N^{\frac12}\sup_{t\in[0,T]}\|u(t)\|_1
			\lesssim |\alpha| N^{\frac32}.
		\end{equation*}
		Hence
		\begin{equation}\label{E_vx2a}
			\left|\int_0^t\int_{\mathbb{R}}\frac{\mu_{\theta}}{v}u\theta_x u_x\right|
			\lesssim |\alpha| N^{\frac32}.
		\end{equation}
		Apply Cauchy's inequality to get
		\begin{equation}\label{E_R1}
			\left|\int_0^t\int_{\mathbb{R}}\frac{\mu_{\theta}}{v^2}v_x\mathcal{R}_1 \right|
			\leq \epsilon\int_0^t\int_{\mathbb{R}}\frac{\mu\theta v_x^2}{v^3}
			+C(\epsilon)\int_0^t\int_{\mathbb{R}}\frac{\mu_{\theta}^2}{\mu
				\theta v}\mathcal{R}_1^2 .
		\end{equation}
		Since
		\begin{equation*}
			\mathcal{R}_1^2\lesssim
			\mu^2\theta_x^2u_x^2+u^2\kappa_{\theta}^2\theta_x^4
			+u^2\kappa^2\theta_{xx}^2+\mu^2u^2u_x^4+u^2\theta^2u_x^2,
		\end{equation*}
		we have from \eqref{E_basic} that
		\begin{align} \notag
			\int_0^t\int_{\mathbb{R}}\frac{\mu_{\theta}^2\mathcal{R}_1^2}
			{\mu \theta v}
			\lesssim~&
			\int_0^t\int_{\mathbb{R}}\frac{\kappa\theta_x^2}{v\theta^2}
			\frac{\mu_{\theta}^2\theta}{\mu \kappa}
			(\mu^2u_x^2+u^2\kappa_{\theta}^2\theta_x^2)
			+
			\int_0^t\int_{\mathbb{R}}\frac{\kappa\theta_{xx}^2}{v}
			\frac{\mu_{\theta}^2\kappa}{\mu \theta}u^2\\ \notag
			&+
			\int_0^t\int_{\mathbb{R}}\frac{\mu u_x^2}{v\theta}
			\frac{\mu_{\theta}^2}{\mu^2}(u^2\theta^2+\mu^2u^2u_x^2)\\
			\notag
			\lesssim~&
			\lrn\frac{\mu_{\theta}^2\theta}{\mu \kappa}
			(\mu^2u_x^2+u^2\kappa_{\theta}^2\theta_x^2)\rrn
			+\lrn\frac{\mu_{\theta}^2\kappa}{\mu \theta}u^2\rrn
			\int_0^t\int_{\mathbb{R}}\frac{\kappa\theta_{xx}^2}{v}\\
			\label{E_R1a}
			&+\lrn\frac{\mu_{\theta}^2}{\mu^2}(u^2\theta^2+\mu^2u^2u_x^2)\rrn.
		\end{align}
		The a priori assumption \eqref{apriori1} implies
		\begin{equation}\label{E_key1}
			\lrn \theta^{\alpha}+\theta^{-\alpha}\rrn\leq
			m_2^{-|\alpha|}+(4N)^{|\alpha|}.
		\end{equation}
		Then we have from \eqref{apriori3}, \eqref{E_key1} and
		Sobolev's inequality that
		\begin{align} \notag
			&\lrn\frac{\mu_{\theta}^2\theta}{\mu \kappa}
			\left(\mu^2u_x^2+u^2\kappa_{\theta}^2\theta_x^2\right)\rrn
			+\lrn\frac{\mu_{\theta}^2\kappa}{\mu \theta}u^2\rrn\\
			&\qquad \label{E_R1b}
			+\lrn\frac{\mu_{\theta}^2}{\mu^2}\left(u^2\theta^2+\mu^2u^2u_x^2\right)\rrn
			\lesssim \alpha^2
			\Xi(m_1,m_2,N)^{\frac{1}{8}}.
		\end{align}
		Combine the estimates \eqref{E_R1}--\eqref{E_R1b}
		to derive
		\begin{align} \notag
			&\left|\int_0^t\int_{\mathbb{R}}\frac{\mu_{\theta}}
			{v^2}v_x\mathcal{R}_1\right|\\
			&\qquad \label{E_R1c}
			\leq \epsilon\int_0^t\int_{\mathbb{R}}\frac{\mu\theta v_x^2}{v^3}
			+C(\epsilon)\alpha^2 \Xi(m_1,m_2,N)^{\frac{1}{8}}\left[1
			+\int_0^t\int_{\mathbb{R}}\frac{\kappa\theta_{xx}^2}{v}\right].
		\end{align}
		For the last term on the right-hand side of \eqref{E_vx2}, we have
		\begin{equation} \label{E_R2}
			\left|\int_0^t\int_{\mathbb{R}}\frac{\mu_{\theta}}{v^2}v_x^2\mathcal{R}_2\right|
			\leq \lrn\frac{\mu_{\theta} v}{\mu\theta}\mathcal{R}_2\rrn
			\int_0^t\int_{\mathbb{R}}\frac{\mu\theta v_x^2 }{v^3}.
		\end{equation}
		It follows from \eqref{apriori3} and \eqref{id_theta1} that
		\begin{equation}\label{pro2.1}
			|\theta_t|\leq \Xi(m_1,m_2,N)^{\frac{1}{8}}
			\left(u_x^2+\theta_x^2+|\theta_x v_x|+|\theta_{xx}|+|u_x|\right)
		\end{equation}
		and
		\begin{equation*}
			|\mathcal{R}_2|\leq
			\Xi(m_1,m_2,N)^{\frac{1}{8}}\left(u_x^2+\theta_x^2+|
			\theta_x v_x|+|\theta_{xx}|+|u_x|+|u\theta_x|\right).
		\end{equation*}
		Hence
		\begin{equation*}
			\lrn\frac{\mu_{\theta} v}{\mu\theta}\mathcal{R}_2\rrn
			\leq |\alpha|\Xi(m_1,m_2,N)^{\frac{1}{4}}.
		\end{equation*}
		We plug this last estimate into \eqref{E_R2} to find that
		\begin{equation} \label{E_R2a}
			\left|\int_0^t\int_{\mathbb{R}}\frac{\mu_{\theta}}{v^2}v_x^2\mathcal{R}_2\right|
			\leq |\alpha|\Xi(m_1,m_2,N)^{\frac{1}{4}}
			\int_0^t\int_{\mathbb{R}}\frac{\mu\theta v_x^2 }{v^3}.
		\end{equation}
		Plugging \eqref{E_vx2a}, \eqref{E_R1c}, and \eqref{E_R2a}
		into \eqref{E_vx2} and choosing $\epsilon>0$ sufficiently small,
		we derive
		\begin{align} \notag
			\left\|\frac{\mu v_x}{v}(t)\right\|^2
			+\int_0^t\int_{\mathbb{R}}\frac{\mu\theta v_x^2}{v^3}
			\lesssim\,&  1+\int_0^t\int_{\mathbb{R}}\frac{\mu u_x^2}{v}
			+\int_0^t\int_{\mathbb{R}}\frac{\mu \theta_x^2}{v\theta}\\
			&\! \notag
			+\alpha^2
			\Xi(m_1,m_2,N)^{\frac{1}{8}}\left[1
			+\int_0^t\int_{\mathbb{R}}\frac{\kappa\theta_{xx}^2}{v}\right]\\
			&\! \label{E_vx3}
			+|\alpha|\Xi(m_1,m_2,N)^{\frac{1}{4}}
			\int_0^t\int_{\mathbb{R}}\frac{\mu\theta v_x^2 }{v^3}.
		\end{align}
		If  the parameter  $\epsilon_1$ in \eqref{apriori3}
		is chosen to be  suitably small,
		then we obtain
		\begin{equation}\label{E_vx4}
			\left\|\frac{\mu v_x}{v}(t)\right\|^2
			+\int_0^t\int_{\mathbb{R}}\frac{\mu\theta v_x^2}{v^3}
			\lesssim 1+\int_0^t\int_{\mathbb{R}}\frac{\mu u_x^2}{v}
			+\int_0^t\int_{\mathbb{R}}\frac{\mu\theta_x^2}{v\theta}
			+|\alpha|\int_0^t\int_{\mathbb{R}}\frac{\kappa\theta_{xx}^2}{v }.
		\end{equation}
		We next estimate the last term in \eqref{E_vx4}.
		To this end, we multiply \eqref{id_theta1} by $\theta_{xx}$ to get
		\begin{equation*}
			\begin{split}
				&\left(\frac{c_v}{2}\theta_x^2\right)_t
				-\left(c_v\theta_x\theta_t\right)_x
				+\frac{\kappa\theta_{xx}^2}{v}\\
				&\qquad
				=\theta_{xx}\left[\frac{\theta u_x}{v}-\frac{\kappa_v v_x\theta_x}{v}
				+\frac{\kappa\theta_x v_x}{v^2}-\frac{\kappa_{\theta}\theta_x^2}{v}
				-\frac{\mu u_x^2}{v}\right].
			\end{split}
		\end{equation*}
		Integrating this last identity over $[0,t]\times\mathbb{R}$ and
		employing Cauchy's inequality give us
		\begin{align} \notag
			\|\theta_x(t)\|^2+\int_0^t\int_{\mathbb{R}}\frac{\kappa\theta_{xx}^2}{v}
			\lesssim&~1+\int_0^t\int_{\mathbb{R}}
			\left[\frac{\theta^2 u_x^2}{v\kappa}+\frac{\kappa_v^2 v_x^2\theta_x^2}{v\kappa}
			+\frac{\kappa\theta_x^2 v_x^2}{v^3}+\frac{\kappa_{\theta}^2\theta_x^4}{v\kappa}
			+\frac{\mu^2 u_x^4}{v\kappa}\right]\\ \notag
			\lesssim&~1+\lrn\frac{v\theta}{\mu}\left(\frac{\theta^2}{v\kappa}
			+\frac{\mu^2u_x^2}{v\kappa}\right)\rrn\int_0^t\int_{\mathbb{R}}\frac{\mu u_x^2}{v\theta}\\ \label{pro2.2}
			&+\lrn\frac{v\theta^2}{\kappa}\left(\frac{\kappa_v^2v_x^2}{v\kappa}
			+\frac{\kappa v_x^2}{v^3}+\frac{\kappa_{\theta}^2\theta_x^2}{v\kappa}\right)\rrn
			\int_0^t\int_{\mathbb{R}}\frac{\kappa \theta_x^2}{v\theta^2}.
		\end{align}
		It follows from  \eqref{apriori3} that
		\begin{equation*}
			\lrn\frac{v\theta}{\mu}\left(\frac{\theta^2}{v\kappa}
			+\frac{\mu^2u_x^2}{v\kappa}\right)\rrn
			+\lrn\frac{v\theta^2}{\kappa}\left(\frac{\kappa_v^2v_x^2}{v\kappa}
			+\frac{\kappa v_x^2}{v^3}+\frac{\kappa_{\theta}^2\theta_x^2}{v\kappa}\right)\rrn
			\lesssim \Xi(m_1,m_2,N)^{\frac{1}{8}}.
		\end{equation*}
		Insert the above inequality and \eqref{E_basic} into \eqref{pro2.2} to derive
		\begin{equation}\label{pro2.2a}
			\|\theta_x(t)\|^2+\int_0^t\int_{\mathbb{R}}\frac{\kappa\theta_{xx}^2}{v}
			\lesssim 1+\Xi(m_1,m_2,N)^{\frac{1}{8}}.
		\end{equation}
		Combining \eqref{E_vx4} and \eqref{pro2.2a},  we obtain
		\begin{align} \notag
			&\left\|\frac{\mu v_x}{v}(t)\right\|^2
			+\int_0^t\int_{\mathbb{R}}\frac{\mu\theta v_x^2}{v^3}\\
			&\qquad \label{E_vx5}
			\lesssim  1+\int_0^t\int_{\mathbb{R}}\frac{\mu u_x^2}{v}
			+\int_0^t\int_{\mathbb{R}}\frac{\mu\theta_x^2}{v\theta}
			+|\alpha|\Xi(m_1,m_2,N)^{\frac{1}{8}}.
		\end{align}
		Under the assumptions \eqref{apriori3},
		we take $\epsilon_1>0$ small enough to infer
		\begin{equation}\label{E_vx6}
			\left\|\frac{\mu v_x}{v}(t)\right\|^2
			+\int_0^t\int_{\mathbb{R}}\frac{\mu\theta v_x^2}{v^3}
			\lesssim 1+\int_0^t\int_{\mathbb{R}}\left[\frac{\mu u_x^2}{v}
			+\frac{\mu\theta_x^2}{v\theta}\right].
		\end{equation}
		The estimate \eqref{E_basic} implies
		\begin{equation}\label{E_vx6a}
			\begin{aligned}
				\int_0^t\int_{\mathbb{R}}\left[
				\frac{\mu u_x^2}{v}+\frac{\mu\theta_x^2}{v\theta}\right]
				\leq\lrn\theta\rrn\int_0^t\int_{\mathbb{R}}\frac{\mu u_x^2}{v\theta}
				+\lrn\frac{\mu \theta}{\kappa}\rrn\int_0^t\int_{\mathbb{R}}\frac{\kappa\theta_x^2}{v\theta^2}
				\lesssim \lrn\theta\rrn.
			\end{aligned}
		\end{equation}
		Plugging \eqref{E_vx6a} into \eqref{E_vx6} yields \eqref{E_vx}.
		This completes the proof of this lemma.
	\end{proof}
	In the next lemma, we apply the technique developed by Kanel$'$\cite{Ka68MR0227619,LYZZMR3225502}
	to estimate the upper and lower bounds
	for the specific volume $v(t,x)$ in terms of $\lrn\theta\rrn$.
	\begin{lemma}
		\label{L_v2}
		Assume that the conditions listed in Lemma \ref{L_v1} hold.
		Then
		\begin{equation}\label{bound_v00}
			\lrn v\rrn\lesssim 1+\lrn\theta\rrn^{\frac1{1+2{\ell_1}}},
			\qquad
			\lrn v^{-1}\rrn\lesssim 1+\lrn\theta\rrn^{\frac1{2{\ell_2}}}.
		\end{equation}
	\end{lemma}
	\begin{proof}
		Define
		$$
		\Phi(v):=\int_1^v\frac{\sqrt{\phi(z)}}{z}h(z)\mathrm{d}z.
		$$
		We infer from \eqref{h} that
		for suitably large constant $C$ and $v\geq C$,
		\begin{equation*}
			\begin{aligned}
				\Phi(v) \gtrsim \int_C^v\frac{\sqrt{\phi(z)}}{z}z^{{\ell_1}}
				\mathrm{d}z
				\gtrsim\int_C^vz^{-\frac12+{\ell_1}}\mathrm{d}z.
			\end{aligned}
		\end{equation*}
		Hence
		$$v^{\frac12+{\ell_1}}\lesssim 1+|\Phi(v)|\quad {\rm for\ all}\ v\in(0,\infty). $$
		Similarly, it follows that
		$$v^{-{\ell_2}}\lesssim 1+|\Phi(v)|\quad {\rm for\ all}\ v\in(0,\infty).$$
		Thus, we have
		\begin{equation}\label{kanel1}
			\lrn v\rrn^{\frac12+{\ell_1}}+\lrn v^{-1}\rrn^{{\ell_2}}\lesssim 1
			+\sup_{(t,x)\in[0,T]\times\mathbb{R}}|\Phi(v(t,x))|.
		\end{equation}
		On the other hand,  the a priori assumption $v-1\in C([0,T];H^3)
		$ implies
		\begin{equation*}\label{kanel2}
			\begin{aligned}
				|\Phi(v)(t,x)|=~&\left|\int_x^\infty\frac{\partial}{\partial x}\Phi(v(t,y))
				\mathrm{d}y\right|\\
				\leq~&\int_{\mathbb{R}}\sqrt{\phi(v(t,y))}\left|\left(\frac{h(v)}{v}v_x\right)(t,y)\right|
				\mathrm{d}y\\
				\leq~&\left\|\sqrt{\phi(v(t))}\right\|\left\|\left(\frac{h(v)v_x}{v}\right)(t)\right\|,
			\end{aligned}
		\end{equation*}
		which combined with Lemmas \ref{L_bas}--\ref{L_v1} and the conditions \eqref{apriori3}
		yields
		\begin{equation}\label{kanel3}
			\begin{split}
				|\Phi(v)(t,x)| \lesssim\left\|\left(\theta^{-\alpha}\frac{\mu v_x}{v}\right)(t)\right\|
				\lesssim 1+\lrn\theta\rrn^{\frac12}.
			\end{split}
		\end{equation}
		Combine \eqref{kanel1} and \eqref{kanel3} to deduce \eqref{bound_v00}.
		The proof is completed.
	\end{proof}
	
	\subsection{Pointwise bounds on temperature}\label{sec_theta}
	This part is devoted to
	obtaining
	pointwise upper and lower bounds of the temperature $\theta(t,x)$
	as well as
	the estimates on
	$H_x^1$-norm of $(v(t,x)-1,u(t,x),\theta(t,x)-1)$.
	We first consider the estimate on the
	$L_x^2(\mathbb{R})$-norm of $(\theta(t,x)-1)$ in the following lemma.
	\begin{lemma}
		\label{L_th1}
		Assume that the conditions listed in Lemma \ref{L_v1} hold. Then
		\begin{equation} \label{E_theta1}
			\begin{aligned}
				\sup_{t\in[0,T]}\left[
				\left\|(\theta-1)(t)\right\|^2
				+\left\|u(t)\right\|^4_{L^4}\right]
				+\int_0^T\int_{\mathbb{R}}\left[\frac{\kappa\theta_x^2}{v}
				+\theta\frac{\mu u_x^2}{v}+\frac{\mu u^2u_x^2}{v}\right]
				\lesssim 1.
			\end{aligned}
		\end{equation}
	\end{lemma}
	\begin{proof}
		For each $t\geq0$ and $a>1$, define
		$$\Omega_a(t):=\{x\in\mathbb{R}:\ \theta(t,x)>a\}.$$
		Multiply \eqref{id_theta} by $(\theta-2)_+:=\max\{\theta-2,0\}$,
		and integrate the resulting identity over $[0,t]\times\mathbb{R}$ to find
		\begin{align}\notag
			&\frac{c_v}2\int_{\mathbb{R}}(\theta-2)_+^2\mathrm{d}x
			-\frac{c_v}2\int_{\mathbb{R}}(\theta_0-2)_+^2\mathrm{d}x
			+\int_0^t\int_{\Omega_2(\tau)}\frac{\kappa\theta_x^2}{v}\\
			&\qquad \label{E_theta1a}
			=
			-\int_0^t\int_{\mathbb{R}}\frac{\theta u_x}{v}(\theta-2)_+
			+\int_0^t\int_{\mathbb{R}}\frac{\mu u_x^2}{v}(\theta-2)_+.
		\end{align}
		To estimate the last term in this last identity,
		we multiply $\eqref{NS_L}_2$  by $2u(\theta-2)_+$
		and then integrate the resulting identity over $[0,t]\times\mathbb{R}$ to infer
		\begin{align} \notag
			&\int_{\mathbb{R}} u^2(\theta-2)_+\mathrm{d}x
			-\int_{\mathbb{R}}u_0^2(\theta_0-2)_+\mathrm{d}x
			+2\int_0^t\int_{\mathbb{R}}\frac{\mu u_x^2}{v}(\theta-2)_+ \\
			&\qquad \label{pro3.1}
			=2\int_0^t\int_{\mathbb{R}}\frac{\theta}{v}u_x(\theta-2)_+
			+\int_0^t\int_{\Omega_2(\tau)}
			\left[2\frac{ \theta}{v}u\theta_x
			-2\frac{\mu u u_x}{v}\theta_x
			+ u^2\theta_t\right].
		\end{align}
		Combine \eqref{E_theta1a} and \eqref{pro3.1} to get
		\begin{align} \notag
			&\int_{\mathbb{R}}\left[\frac{c_v}2(\theta-2)_+^2+u^2(\theta-2)_+\right]\mathrm{d}x
			+\int_0^t\int_{\Omega_2(\tau)}\left[\frac{\kappa\theta_x^2}{v}
			+\frac{\mu u_x^2}{v}(\theta-2)_+\right] \\
			&\qquad \label{E_theta1b}
			=\int_{\mathbb{R}} \left[\frac{c_v}2(\theta_0-2)_+^2+u_0^2(\theta_0-2)_+\right]\mathrm{d}x
			+\sum_{p=1}^5\mathcal{J}_p,
		\end{align}
		where each term $\mathcal{J}_p$ in the decomposition will be defined below.
		First, we consider the term
		\begin{equation*}
			\mathcal{J}_1:=\int_0^t\int_{\mathbb{R}}\frac{\theta}{v}(\theta-2)_+u_x.
		\end{equation*}
		We deduce from the condition \eqref{h} that
		\begin{equation}\label{E_h1}
			\lrn h(v)^{-1}v^{-1}\rrn+\lrn h(v)^{-1}v\rrn \lesssim 1,
		\end{equation}
		which along with \eqref{apriori3} implies
		\begin{equation}\label{E_h2}
			\lrn \mu^{-1}v^{-1}\rrn+\lrn \mu^{-1} v\rrn \lesssim 1.
		\end{equation}
		It follows from Cauchy's inequality and \eqref{E_h2} that
		\begin{align}\notag
			|\mathcal{J}_1|
			\leq&~ \epsilon\int_0^t\int_{\mathbb{R}}\frac{\mu u_x^2}{v}(\theta-2)_+
			+C(\epsilon)\lrn\mu^{-1}v^{-1}\rrn
			\int_0^t\int_{\Omega_{2}(\tau)}\theta^2(\theta-2)_+\\ \label{J1}
			\leq&~ \epsilon\int_0^t\int_{\mathbb{R}}\frac{\mu u_x^2}{v}(\theta-2)_+
			+C(\epsilon)\int_0^t\sup_{x\in\mathbb{R}}\left(\theta-\tfrac{3}{2}\right)_+^2\mathrm{d}\tau.
		\end{align}
		Here we have used
		\begin{equation} \label{key1}
			\int_{\Omega_{2}(\tau)}\theta\mathrm{d}x
			\lesssim \int_{\mathbb{R}}\phi(\theta)\mathrm{d}x\lesssim 1.
		\end{equation}
		For the terms
		\begin{equation*}
			\mathcal{J}_2:=2\int_0^t\int_{\Omega_2(\tau)}\frac{\theta}{v}u\theta_x
			\quad{\rm and}\quad
			\mathcal{J}_3:=-2\int_0^t\int_{\Omega_2(\tau)}\frac{\mu uu_x}{v}\theta_x,
		\end{equation*}
		we derive from Cauchy's inequality, \eqref{E_basic}, and \eqref{E_h2}  that
		\begin{align} \notag
			|\mathcal{J}_2|
			\leq&~\epsilon\int_0^t\int_{\Omega_2(\tau)}\frac{\kappa\theta_x^2}{v}
			+C(\epsilon)\lrn\kappa^{-1}v^{-1}\rrn\int_0^t\int_{\Omega_2(\tau)}\theta^2u^2\\ \label{J2}
			\leq&~\epsilon\int_0^t\int_{\Omega_2(\tau)}\frac{\kappa\theta_x^2}{v}
			+C(\epsilon)\int_0^t\sup_{x\in\mathbb{R}}\left(\theta-\tfrac{3}{2}\right)_+^2\mathrm{d}\tau,
		\end{align}
		and
		\begin{equation}\label{J3}
			\begin{aligned}
				|\mathcal{J}_3|
				\leq&~\epsilon\int_0^t\int_{\Omega_2(\tau)}\frac{\kappa\theta_x^2}{v}
				+C(\epsilon)\int_0^t\int_{\Omega_2(\tau)}\frac{\mu u^2u_x^2}{v}.
			\end{aligned}
		\end{equation}
		For the term
		\begin{equation*}
			\mathcal{J}_4:=\frac1{c_v}\int_0^t\int_{\Omega_2(\tau)}u^2
			\left[\mu\frac{u_x^2}v-\frac{\theta}vu_x\right],
		\end{equation*}
		similar to the estimate for $\mathcal{J}_2$, we have
		\begin{align}\notag
			|\mathcal{J}_4| 
			\lesssim&~\int_0^t\int_{\Omega_2(\tau)}\frac{\mu u^2u_x^2}{v}
			+\lrn\mu^{-1}v^{-1}\rrn\int_0^t\int_{\Omega_2(\tau)}\theta^2u^2\\
			\label{J4}
			\lesssim&~\int_0^t\int_{\Omega_2(\tau)}\frac{\mu u^2u_x^2}{v}
			+\int_0^t\sup_{x\in\mathbb{R}}\left(\theta-\tfrac{3}{2}\right)_+^2\mathrm{d}\tau.
		\end{align}
		For the last term
		$$\mathcal{J}_5
		:=\int_0^t\int_{\Omega_2(\tau)}\frac{1}{c_v}u^2
		\left(\frac{\kappa\theta_x}v\right)_x,$$
		we apply Lebesgue's dominated convergence theorem to find
		\begin{equation*}
			\begin{aligned}
				\mathcal{J}_5
				=&~\int_0^t\int_{\mathbb{R}}\lim_{\nu\rightarrow 0^+}
				\varphi_\nu(\theta)\frac{u^2}{c_v}\left(\frac{\kappa\theta_x}v\right)_x\\
				=&~\frac{1}{c_v}\lim_{\nu\rightarrow 0^+}\int_0^t
				\int_{\mathbb{R}}\left[-2\varphi_\nu(\theta)uu_x
				\frac{\kappa\theta_x}v-\varphi_\nu'(\theta)u^2\frac{\kappa\theta_x^2}{v}\right],
			\end{aligned}
		\end{equation*}
		where $\varphi_\nu$ is defined by
		\begin{equation*}
			\varphi_\nu(\theta):=
			\left\{
			\begin{array}{rcl}
				&1,  & \quad {\theta-2\geq\nu,}\\
				&(\theta-2)/\nu,  &\qquad { 0\leq\theta-2<\nu,}\\
				&0,   &\qquad {\theta-2<0.}
			\end{array}\right.
		\end{equation*}
		Hence 
		\begin{align} \notag
			\mathcal{J}_5
			\leq&~-\frac{2}{c_v}\lim_{\nu\rightarrow 0^+}\int_0^t
			\int_{\mathbb{R}}\varphi_\nu(\theta)uu_x\frac{\kappa\theta_x}v\\
			\notag
			\leq&~-\frac{2}{c_v}\int_0^t
			\int_{\Omega_2(\tau)}\frac{\kappa uu_x\theta_x}v\\
			\label{J5}
			\leq&~\epsilon\int_0^t\int_{\Omega_2(\tau)}\frac{\kappa\theta_x^2}{v}
			+C(\epsilon)\int_0^t\int_{\Omega_2(\tau)}\frac{\mu u^2u_x^2}{v}.
		\end{align}
		Plug the estimates \eqref{J1}--\eqref{J5} into \eqref{E_theta1b} to infer
		\begin{align} \notag
			&\int_{\mathbb{R}}(\theta-2)_+^2\mathrm{d}x
			+\int_0^t\int_{\Omega_2(\tau)}\left[\frac{\kappa\theta_x^2}{v}
			+\frac{\mu u_x^2}{v}(\theta-2)_+\right] \\
			&\qquad \label{E_theta1c}
			\lesssim
			1+
			\int_0^t\sup_{x\in\mathbb{R}}\left(\theta-\tfrac{3}{2}\right)_+^2\mathrm{d}\tau
			+\int_0^t\int_{\Omega_2(\tau)}\frac{\mu u^2u_x^2}{v}.
		\end{align}
		
		It follows from \eqref{E_basic} that
		\begin{align} \notag
			\int_0^t\int_{\mathbb{R}}\frac{\kappa\theta_x^2}{v}
			=&~\int_0^t\left[\int_{\Omega_2(\tau)}+\int_{\mathbb{R}\setminus\Omega_2(\tau)}\right]
			\frac{\kappa\theta_x^2}{v}\\ \notag
			\leq&~\int_0^t\int_{\Omega_2(\tau)}\frac{\kappa \theta_x^2}{v}
			+C\int_0^t\int_{\mathbb{R}}\frac{\kappa \theta_x^2}{v\theta^2}\\
			\label{pro3.5}
			\lesssim&~1+\int_0^t\int_{\Omega_2(\tau)}\frac{\kappa \theta_x^2}{v},
		\end{align}
		and
		\begin{align} \notag
			\int_0^t\int_{\mathbb{R}}\theta\frac{\mu u_x^2}{v}
			=&~\int_0^t\left[\int_{\Omega_3(\tau)}+\int_{\mathbb{R}\setminus\Omega_3(\tau)}\right]
			\theta\frac{\mu u_x^2}{v}\\ \notag
			\lesssim&~\int_0^t\int_{\Omega_2(\tau)}\frac{\mu u_x^2}{v}(\theta-2)_+
			+\int_0^t\int_{\mathbb{R}}\frac{\mu u_x^2}{v\theta}\\ \label{pro3.6}
			\lesssim&~1+\int_0^t\int_{\Omega_2(\tau)}\frac{\mu u_x^2}{v}(\theta-2)_+.
		\end{align}
		Insert \eqref{pro3.5} and \eqref{pro3.6} into \eqref{E_theta1c} to discover
		\begin{align} \notag
			&\int_{\mathbb{R}}(\theta-2)_+^2\mathrm{d}x
			+\int_0^t\int_{\mathbb{R}}\left[\frac{\kappa\theta_x^2}{v}
			+\theta\frac{\mu u_x^2}{v}\right] \\ \label{E_theta1d}
			&\qquad
			\lesssim 1+
			\int_0^t\sup_{x\in\mathbb{R}}\left(\theta-\tfrac{3}{2}\right)_+^2\mathrm{d}\tau
			+\int_0^t\int_{\Omega_2(\tau)}\frac{\mu u^2u_x^2}{v}.
		\end{align}
		
		In order to estimate the last term in \eqref{E_theta1d},
		we multiply $\eqref{NS_L}_2$ by $u^3$ to have
		\begin{equation*}
			\left(\frac14u^4\right)_t
			+ \left[u^3\left(\frac{\theta}{v}-1\right)-u^3\frac{\mu u_x}{v}\right]_x
			=3u^2u_x\left[\frac{\theta}{v}-1-\frac{\mu u_x}{v}\right].
		\end{equation*}
		Integrate the above identity over $[0,t]\times\mathbb{R}$ to obtain
		\begin{equation}\label{pro3.7}
			\int_{\mathbb{R}}u^4\mathrm{d}x
			+\int_0^t\int_{\mathbb{R}}\frac{\mu u^2u_x^2}{v}
			\lesssim 1+ \sum_{p=1}^4\mathcal{I}_p,
		\end{equation}
		where each term $\mathcal{I}_p$ in the decomposition will be defined and estimated  as follows.
		First we consider the term
		\begin{equation*}
			\mathcal{I}_1:=\int_0^t\int_{\Omega_2(\tau)}u^2u_x\frac{\theta-1}{v}.
		\end{equation*}
		Applying Cauchy's inequality, \eqref{E_basic}, and \eqref{E_h2}, we get
		\begin{align} \notag
			|\mathcal{I}_1|
			\leq&~\nu\int_0^t\int_{\mathbb{R}}\frac{\mu u^2u_x^2}{v}
			+C(\nu)\lrn\mu^{-1}v^{-1}\rrn\int_0^t\int_{\Omega_2(\tau)}(\theta-1)^2u^2\\ \label{I1}
			\leq&~\nu\int_0^t\int_{\mathbb{R}}\frac{\mu u^2u_x^2}{v}
			+C(\nu)\int_0^t\sup_{x\in\mathbb{R}}\left(\theta-\tfrac{3}{2}\right)_+^2\mathrm{d}\tau.
		\end{align}
		For the terms
		\begin{equation*}
			\mathcal{I}_2 :=\int_0^t\int_{\mathbb{R}\setminus\Omega_2(\tau)}
			u^2u_x\frac{\theta-1}{v}\quad {\rm and}\quad
			\mathcal{I}_3 :=\int_0^t\int_{v\leq 2} u^2u_x\frac{1-v}{v},
		\end{equation*}
		we have from \eqref{E_phi} and \eqref{E_basic} that
		\begin{equation*}
			\int_{\mathbb{R}\setminus\Omega_2(\tau)}(\theta-1)^2\mathrm{d}x
			+\int_{v(\tau,x)\leq 2}(v-1)^2\mathrm{d}x
			\lesssim \int_{\mathbb{R}}\eta(v,u,\theta)\mathrm{d}x
			\lesssim 1.
		\end{equation*}
		In view of H\"{o}lder's inequality and \eqref{E_basic}, we deduce
		\begin{equation*}
			\mathcal{I}_{2}+\mathcal{I}_{3}
			\lesssim\int_0^t\|u \|_{L^\infty}^2\left\|\frac{u_x}{v} \right\|
			\lesssim\int_0^t\|u_x \|\left\|\frac{u_x}{v} \right\|
			\lesssim \int_0^t\int_{\mathbb{R}}u_x^2
			+\int_0^t\int_{\mathbb{R}}\frac{u_x^2}{v^2}.
		\end{equation*}
		Applying Cauchy's inequality again,
		we infer from \eqref{E_basic} and \eqref{E_h2} that
		\begin{align}\notag
			\mathcal{I}_{2}+\mathcal{I}_{3}
			&\lesssim \epsilon\int_0^t\int_{\mathbb{R}}\frac{\mu\theta u_x^2}{v}
			+C(\epsilon)\int_0^t\int_{\mathbb{R}}
			\left[\frac{v u_x^2}{\mu \theta}+\frac{u_x^2}{v^3\mu\theta}\right]\\
			\notag
			&\lesssim \epsilon\int_0^t\int_{\mathbb{R}}\frac{\mu\theta u_x^2}{v}
			+C(\epsilon)\lrn\mu^{-2}v^{2}+\mu^{-2}v^{-2}\rrn
			\int_0^t\int_{\mathbb{R}}\frac{\mu u_x^2}{v\theta}\\ \label{I23}
			&\lesssim \epsilon\int_0^t\int_{\mathbb{R}}\frac{\mu\theta u_x^2}{v}
			+C(\epsilon).
		\end{align}
		Let us now consider the term
		\begin{equation*}
			\mathcal{I}_4:=\int_0^t\int_{v\geq 2}u^2u_x\frac{1-v}{v}.
		\end{equation*}
		In view of \eqref{E_phi},
		we obtain that $\left|\frac{1-v}{v}\right|\lesssim\sqrt{\phi(v)}$ for all $v\geq 2$,
		which combined with
		\eqref{E_basic} implies
		\begin{equation*}
			\begin{aligned}
				\mathcal{I}_{4}\lesssim \int_0^t\int_{v\geq 2}u^2|u_x|\sqrt{\phi(v)}
				\lesssim \int_0^t\|u\|_{L^\infty}^2\|u_x \|\left\|\sqrt{\phi(v)}\right\|
				\lesssim \int_0^t\|u_x\|^2.
			\end{aligned}
		\end{equation*}
		Hence we have from \eqref{E_basic} and \eqref{E_h2} that
		\begin{equation}\label{I4}
			\mathcal{I}_{4}
			\lesssim\epsilon\int_0^t\int_{\mathbb{R}}\frac{\mu\theta u_x^2}{v}
			+C(\epsilon)\lrn\mu^{-2}v^2\rrn\int_0^t\int_{\mathbb{R}}
			\frac{\mu u_x^2}{v \theta}
			\lesssim \epsilon\int_0^t\int_{\mathbb{R}}\frac{\mu\theta u_x^2}{v}
			+C(\epsilon).
		\end{equation}
		Insert \eqref{I1}--\eqref{I4} into \eqref{pro3.7} and let $\nu>0$ suitably small to derive
		\begin{equation}\label{pro3.7a}
			\begin{aligned}
				\int_{\mathbb{R}}u^4\mathrm{d}x
				+\int_0^t\int_{\mathbb{R}}\frac{\mu u^2u_x^2}{v}
				\lesssim
				C(\epsilon)+\epsilon\int_0^t\int_{\mathbb{R}}\frac{\mu\theta u_x^2}{v}
				+\int_0^t\sup_{x\in\mathbb{R}}\left(\theta-\tfrac{3}{2}\right)_+^2\mathrm{d}\tau.
			\end{aligned}
		\end{equation}
		Combining \eqref{E_theta1d} and \eqref{pro3.7a},
		we take $\epsilon>0$ small enough to get
		\begin{align}\notag
			&\int_{\mathbb{R}}\left[(\theta-2)_+^2+u^4\right]\mathrm{d}x
			+\int_0^t\int_{\mathbb{R}}\left[\frac{\kappa\theta_x^2}{v}
			+\theta\frac{\mu u_x^2}{v}+\frac{\mu u^2u_x^2}{v}\right]\\
			\label{E_theta1e} &\qquad \lesssim 1+
			\int_0^t\sup_{x\in\mathbb{R}}\left(\theta-\tfrac{3}{2}\right)_+^2\mathrm{d}\tau.
		\end{align}
		In light of the fundamental theorem of calculus and \eqref{key1}, we infer
		\begin{align}\notag
			\int_0^t\sup_{x\in\mathbb{R}}\left(\theta-\tfrac{3}{2}\right)_+^2\mathrm{d}\tau
			\lesssim&~ \int_0^t\int_{\Omega_{3/2}(\tau)}\frac{\theta_x^2}\theta\\
			\notag
			\lesssim&~\delta\int_0^t\int_{\mathbb{R}}\frac{\kappa\theta_x^2}v
			+C(\delta)\int_0^t\int_{\mathbb{R}}\frac{v\theta_x^2}{\kappa\theta^2}\\ \notag
			\lesssim&~\delta\int_0^t\int_{\mathbb{R}}\frac{\kappa\theta_x^2}v
			+C(\delta)\int_0^t\int_{\mathbb{R}}\frac{\kappa\theta_x^2}{v\theta^2}
			\lrn\frac{v^2}{\kappa^2}\rrn\\ \label{pro3.9}
			\lesssim&~\delta\int_0^t\int_{\mathbb{R}}\frac{\kappa\theta_x^2}v
			+C(\delta).
		\end{align}
		If we plug this last inequality into \eqref{E_theta1e} and
		choose $\delta>0$ sufficiently small,
		then we can deduce \eqref{E_theta1} and hence
		finish the proof of the lemma.
	\end{proof}
	The next lemma concerns the estimate for the first-order derivative 
	with respect to $x$ of $v(t,x)$.
	\begin{lemma}
		Assume that the conditions listed in Lemma \ref{L_v1} hold. Then
		\begin{equation}\label{E_vx'}
			\sup_{t\in[0,T]}\left\|\frac{\mu v_x}{v}(t)\right\|^2
			+\int_0^T\int_{\mathbb{R}}\frac{\mu \theta v_x^2}{v^3}\lesssim 1.
		\end{equation}
	\end{lemma}
	\begin{proof}
		Applying Cauchy's inequality,
		we deduce from \eqref{E_vx6}, \eqref{E_basic} and \eqref{E_theta1} that
		\begin{equation*}
			\begin{aligned}
				\left\|\frac{\mu v_x}{v}(t)\right\|^2
				+\int_0^t\int_{\mathbb{R}}\frac{\mu \theta v_x^2}{v^3}
				\lesssim 1+\int_0^t\int_{\mathbb{R}}
				\left[\frac{\mu\theta u_x^2}{v}+\frac{\mu u_x^2}{v\theta}
				+\frac{\mu \theta_x^2}{v}+\frac{\mu \theta_x^2}{v\theta^2}\right]
				\lesssim 1.
			\end{aligned}
		\end{equation*}
		The proof of the lemma is completed.
	\end{proof}
	For the estimate on the first-order derivative of
	$u(t,x)$, we have
	\begin{lemma}
		Assume that the conditions listed in Lemma \ref{L_v1} hold. Then
		\begin{equation}\label{E_ux}
			\sup_{t\in[0,T]}\|u_x(t)\|^2+\int_0^T\int_{\mathbb{R}}\frac{\mu u_{xx}^2}{v}
			\lesssim 1+\lrn\theta\rrn.
		\end{equation}
	\end{lemma}
	\begin{proof}
		Multiply $\eqref{NS_L}_2$ by $u_{xx}$ to get
		\begin{equation*}
			\left(\frac12u_x^2\right)_t
			-(u_xu_t)_x-\left(\frac{\theta}{v}\right)_xu_{xx}
			=-\frac{\mu u_{xx}^2}{v}
			+\frac{\mu u_{xx}v_xu_x}{v^2}
			-\frac{\mu_xu_xu_{xx}}{v}.
		\end{equation*}
		We integrate the above identity over $[0,t]\times\mathbb{R}$
		and apply Cauchy's inequality to have
		\begin{equation*}
			\begin{aligned}
				\|u_x(t)\|^2+\int_0^t\int_{\mathbb{R}}\frac{\mu u_{xx}^2}{v}
				\lesssim ~&1+
				\lrn\frac{1}{\mu\kappa}\rrn
				\int_0^t\int_{\mathbb{R}}\frac{\kappa\theta_x^2}{v}
				+\lrn\frac{\theta}{\mu^2}\rrn\int_0^t\int_{\mathbb{R}}\frac{\mu\theta v_x^2}{v^3}\\
				&+\lrn\frac{1}{\mu v}\rrn
				\int_0^t
				\left\|u_x\right\|_{L^\infty}^2\left\|\frac{\mu v_x}{v}\right\|^2
				+\int_{0}^t\int_{\mathbb{R}}\frac{\mu_x^2 u_x^2}{\mu v}.
			\end{aligned}
		\end{equation*}
		In view of \eqref{h} and \eqref{apriori3}, we obtain
		\begin{equation*}
			\lrn\frac{1}{\mu\kappa}\rrn+
			\lrn\frac{1}{\mu v}\rrn\lesssim 1,
		\end{equation*}
		which combined with \eqref{E_vx'} and \eqref{E_theta1} yields
		\begin{equation}\label{E_ux1}
			\begin{aligned}
				\|u_x(t)\|^2+\int_0^t\int_{\mathbb{R}}\frac{\mu u_{xx}^2}{v}
				\lesssim 1+\lrn\theta\rrn
				+\int_0^t\left\|{u_x}\right\|_{L^\infty}^2
				+\int_{0}^t\int_{\mathbb{R}}\frac{\mu_x^2 u_x^2}{\mu v}.
			\end{aligned}
		\end{equation}
		We have from \eqref{E_h2} that
		\begin{align} \notag
			\int_0^t\left\|{u_x}\right\|_{L^\infty}^2&\lesssim\int_0^t\|u_x\|\|u_{xx}\|\\ \notag 
			&\lesssim\epsilon\int_0^t\int_{\mathbb{R}}\frac{\mu u_{xx}^2}{v}
			+C(\epsilon))\lrn\frac{v}{\mu}\rrn^2
			\int_0^t\int_{\mathbb{R}}\frac{\mu u_x^2}{v}\\ \label{pro6.4}
			&\lesssim\epsilon\int_0^t\int_{\mathbb{R}}\frac{\mu u_{xx}^2}{v}
			+C(\epsilon).
		\end{align}
		For the last term on the right-hand side of \eqref{E_ux1},
		we use \eqref{apriori3},
		\eqref{h}, \eqref{E_basic}, \eqref{E_vx'}, and \eqref{pro6.4} to discover
		\begin{align} \notag
			\int_{0}^t\int_{\mathbb{R}}\frac{\mu_x^2 u_x^2}{\mu v}
			&\leq \int_0^t\int_{\mathbb{R}}\frac{\mu_v^2v_x^2u_x^2}{\mu v}
			+\int_0^t\int_{\mathbb{R}}\frac{\mu_{\theta}^2{\theta}_x^2u_x^2}{\mu v}\\ \notag
			&\leq \lrn\frac{\mu_v^2 v}{\mu^3}\rrn
			\int_0^t\left\|u_x\right\|_{L^\infty}^2\left\|\frac{\mu v_x}{v}\right\|^2
			+\lrn\frac{\mu_{\theta}^2\theta^2 u_x^2}{\mu}\rrn
			\int_0^t\int_{\mathbb{R}}\frac{\mu\theta_x^2}{v\theta^2}\\
			\notag
			&\leq \lrn\frac{h'(v)^2v}{h(v)^3}\rrn
			\int_0^t\left\|{u_x}\right\|_{L^\infty}^2+
			\alpha^2N^2
			\int_0^t\int_{\mathbb{R}}\frac{\mu\theta_x^2}{v\theta^2}\\ 
			\label{pro6.6}
			&\lesssim \epsilon\int_0^t\int_{\mathbb{R}}\frac{\mu u_{xx}^2}{v}
			+C(\epsilon) +1.
		\end{align}
		Plug \eqref{pro6.4} and \eqref{pro6.6} into \eqref{E_ux1} and
		choose $\epsilon$ small enough to  derive \eqref{E_ux}.
	\end{proof}
	We now turn to deduce an upper bound on the temperature $\theta(t,x)$.
	\begin{lemma}
		\label{L_th2}
		Assume that the conditions listed in Lemma \ref{L_v1} hold. Then
		there exist positive constants $C_i$ ($i=1,2,3$), which
		depend
		only on $\Pi_0$, $V_0$, and $H(V_0)$,
		such that for all $(t,x)\in[0,T]\times\mathbb{R}$,
		\begin{gather}  \label{upper_th}
			\theta(t,x) \leq C_1,\\[1.5mm]
			\label{bound_v}
			C_2\leq v(t,x)\leq C_2^{-1}
			\\  		\label{E_1order} 
			|(v-1,u,\theta-1)(t)\|_1^2
			+\int_0^t\left[\left\|\sqrt{\theta}v_x(s)\right\|^2
			+\|(\theta_x,u_x)(s)\|_1^2\right]\mathrm{d}s\leq C_3^2.
		\end{gather}
	\end{lemma}
	\begin{proof}
		Multiply \eqref{id_theta} by $\theta_{xx}$ and
		integrate the resulting identity to find
		\begin{align} \notag
			&\frac{c_v}{2}\|\theta_{x}(t)\|^2-\frac{c_v}{2}\|\theta_{0x}\|^2
			+\int_0^t\int_{\mathbb{R}}\frac{\kappa\theta_{xx}^2}{v}
			\\  &\qquad \label{E_theta_x}
			=\int_0^t\int_{\mathbb{R}}
			\left[\frac{\kappa v_x\theta_x}{v^2}\theta_{xx}
			+\frac{\theta}{v}u_x\theta_{xx}
			-\frac{\mu u_x^2}{v}\theta_{xx}
			-\frac{\kappa_x\theta_x}{v}\theta_{xx}\right].
		\end{align}
		Next we estimate each term in \eqref{E_theta_x}. First,
		\begin{align} \notag
			\left|\int_0^t\int_{\mathbb{R}}\frac{\kappa v_x\theta_x}{v^2}\theta_{xx}\right|
			&\lesssim  \int_0^t\|\theta_x\|_{L^\infty}
			\left\|\sqrt{\frac{\kappa}{v}}\theta_{xx}\right\|
			\left\|\sqrt{\frac{\kappa}{v}}\frac{v_x}{v}\right\|\\ \notag
			&\lesssim \lrn\frac{1}{\sqrt{\kappa v}}\rrn
			\int_0^t\|\theta_x\|^{\frac12}\|\theta_{xx}\|^{\frac12}
			\left\|\sqrt{\frac{\kappa}{v}}\theta_{xx}\right\|
			\left\|\frac{\kappa v_x}{v}\right\|\\
			&\lesssim  \label{pro7.1}
			\lrn{\frac{v}{\kappa}}\rrn^{\frac14}\lrn\frac{1}{{\kappa v}}\rrn^{\frac12}
			\int_0^t\|\theta_x\|^{\frac12}
			\left\|\sqrt{\frac{\kappa}{v}}\theta_{xx}\right\|^{\frac32}
			\left\|\frac{\kappa v_x}{v}\right\|.
		\end{align}
		In light of Young's inequality,
		we combine \eqref{pro7.1},
		\eqref{E_h2}, \eqref{E_theta1}, and \eqref{E_vx'} to get
		\begin{align} \notag
			\left|\int_0^t\int_{\mathbb{R}}\frac{\kappa v_x\theta_x}{v^2}\theta_{xx}\right|
			&\lesssim \epsilon\int_0^t\int_{\mathbb{R}}\frac{\kappa\theta_{xx}^2}{v}
			+C(\epsilon)
			\int_0^t\|\theta_x\|^2\left\|\frac{\kappa v_x}{v}\right\|^4\\ \notag
			&\lesssim \epsilon\int_0^t\int_{\mathbb{R}}\frac{\kappa\theta_{xx}^2}{v}
			+C(\epsilon)
			\sup_{[0, T]}\left\|\frac{\mu v_x}{v}\right\|^4
			\int_0^t\int_{\mathbb{R}}\frac{\kappa\theta_x^2}{v}\\ 	\label{pro7.1a}
			&\lesssim \epsilon\int_0^t\int_{\mathbb{R}}\frac{\kappa\theta_{xx}^2}{v}
			+C(\epsilon),
		\end{align}
		and
		\begin{align}  \notag
			\left|\int_0^t\int_{\mathbb{R}}\frac{\theta}{v}u_x\theta_{xx}\right|
			&\lesssim \lrn\theta\rrn
			\int_0^t \left\|\sqrt{\frac{\kappa}{v}}\theta_{xx}\right\|
			\left\|\frac{1}{\sqrt{\kappa v}}u_x\right\|\\  \notag
			&\lesssim \epsilon\int_0^t\int_{\mathbb{R}}\frac{\kappa\theta_{xx}^2}{v}
			+C(\epsilon)\lrn\theta\rrn^2\lrn\frac{1}{\kappa\mu}\rrn
			\int_0^t\int_{\mathbb{R}}\frac{\mu u_x^2}{v}\\ \label{pro7.3}
			&\lesssim \epsilon\int_0^t\int_{\mathbb{R}}\frac{\kappa\theta_{xx}^2}{v}
			+C(\epsilon)\lrn\theta\rrn^2.
		\end{align}
		Using H\"{o}lder's inequality, we have from \eqref{E_h2} that
		\begin{equation*}
			\begin{aligned}
				\left|\int_0^t\int_{\mathbb{R}}\frac{\mu u_x^2}{v}\theta_{xx}\right|
				&\lesssim
				\int_0^t\|u_x\|_{L^\infty}
				\left\|\sqrt{\frac{\mu}{v}}u_x\right\|
				\left\|\sqrt{\frac{\mu}{v}}\theta_{xx}\right\|
				\\
				&\lesssim \int_0^t\|u_x\|^{\frac12}\|u_{xx}\|^{\frac12}
				\left\|\sqrt{\frac{\mu}{v}}u_x\right\|
				\left\|\sqrt{\frac{\mu}{v}}\theta_{xx}\right\|
				\\
				&\lesssim
				\sup_{[0, T]}\|u_x \|
				\int_0^t
				\left\|\sqrt{\frac{\mu}{v}}u_{xx}\right\|^{\frac12}
				\left\|\sqrt{\frac{\mu}{v}}u_x\right\|^{\frac{1}{2}}
				\left\|\sqrt{\frac{\mu}{v}}\theta_{xx}\right\|,
			\end{aligned}
		\end{equation*}
		which combined with \eqref{E_basic}, \eqref{E_theta1}, and \eqref{E_ux} implies
		\begin{align} \notag
			\left|\int_0^t\int_{\mathbb{R}}\frac{\mu u_x^2}{v}\theta_{xx}\right|
			&\lesssim  \epsilon\int_0^t\int_{\mathbb{R}}\frac{\kappa\theta_{xx}^2}{v}
			+ C(\epsilon)\sup_{[0, T]}\|u_x \|^2
			\int_0^t
			\int_{\mathbb{R}}\left[\frac{\mu u_{x}^2}{v}+
			\frac{\mu u_{xx}^2}{v}\right]\\ \label{pro7.4}
			&\lesssim  \epsilon\int_0^t\int_{\mathbb{R}}\frac{\kappa\theta_{xx}^2}{v}
			+ C(\epsilon)\left(1+\lrn\theta\rrn\right)^2.
		\end{align}
		For the last term on the right-hand side of \eqref{E_theta_x},  we have
		\begin{equation*}
			\left|\int_0^t\int_{\mathbb{R}}\frac{\kappa_x\theta_x}{v}\theta_{xx}\right|
			\leq\left|\int_0^t\int_{\mathbb{R}}\frac{\kappa_\theta\theta_x^2}{v}\theta_{xx}\right|
			+\left|\int_0^t\int_{\mathbb{R}}\frac{\kappa_vv_x\theta_x}{v}\theta_{xx}\right|.
		\end{equation*}
		We obtain from \eqref{E_basic} and \eqref{apriori3} that
		\begin{align} \notag
			\left|\int_0^t\int_{\mathbb{R}}\frac{\kappa_\theta\theta_x^2}{v}\theta_{xx}\right|
			&\lesssim \epsilon\int_0^t\int_{\mathbb{R}}\frac{\kappa\theta_{xx}^2}{v}
			+C(\epsilon)\alpha^2\lrn\theta_{x}\rrn^2
			\int_0^t\int_{\mathbb{R}}\frac{\kappa\theta_{x}^2}{v\theta^2}\\
			\label{pro7.5}
			&\lesssim \epsilon\int_0^t\int_{\mathbb{R}}\frac{\kappa\theta_{xx}^2}{v}
			+C(\epsilon).
		\end{align}
		In view of \eqref{h}, \eqref{E_h1}, \eqref{E_vx'}, and \eqref{E_theta1}, we infer
		\begin{align} \notag
			\left|\int_0^t\int_{\mathbb{R}}\frac{\kappa_vv_x\theta_x}{v}\theta_{xx}\right|
			&\lesssim\lrn\frac{\kappa_v}{\mu}\sqrt{\frac{v}{\kappa}}\rrn
			\int_0^t\|\theta_x\|_{L^\infty}
			\left\|\sqrt{\frac{\kappa}{v}}\theta_{xx}\right\|
			\left\|\frac{\mu v_x}{v}\right\|\\ \notag
			&\lesssim\int_0^t\left\|\sqrt{\frac{\kappa}{v}}\theta_{xx}\right\|^{\frac32}
			\left\|\sqrt{\frac{\kappa}{v}}\theta_{x}\right\|^{\frac12}
			\left\|\frac{\mu v_x}{v}\right\|\\ \notag
			&\lesssim\epsilon\int_0^t\int_{\mathbb{R}}\frac{\kappa\theta_{xx}^2}{v}
			+C(\epsilon)\int_0^t\left\|\sqrt{\frac{\kappa}{v}}\theta_{x}\right\|^2
			\left\|\frac{\mu v_x}{v}\right\|^4\\  \label{pro7.6}
			&\lesssim  \epsilon\int_0^t\int_{\mathbb{R}}\frac{\kappa\theta_{xx}^2}{v}
			+C(\epsilon).
		\end{align}
		We plug \eqref{pro7.1a}--\eqref{pro7.6} into \eqref{E_theta_x},
		and take $\epsilon>0$ suitably small to derive
		\begin{equation} \label{E_theta_x1}
			\|\theta_{x}(t)\|^2
			+\int_0^t\int_{\mathbb{R}}\frac{\kappa\theta_{xx}^2}{v}
			\lesssim 1+\lrn\theta\rrn^2.
		\end{equation}
		Combining  \eqref{E_theta1} and \eqref{E_theta_x1} gives
		\begin{equation*}
			\lrn\theta\rrn^2=\sup_{t\in[0,T]}\left\|\theta(t)\right\|_{L^{\infty}}^2
			\leq \sup_{t\in[0,T]}
			\left\|\theta(t)\right\|\left\|\theta_x(t)\right\|
			\lesssim 1+\lrn\theta\rrn,
		\end{equation*}
		Apply Cauchy's inequality  to the last inequality to  obtain \eqref{upper_th}.
		We then derive \eqref{bound_v} by plugging \eqref{upper_th} into \eqref{bound_v00}.
		
		Insert \eqref{upper_th} into \eqref{E_ux} and \eqref{E_theta_x1} to give
		\begin{equation} \label{E_1order1}
			\|(u_x,\theta_x)(t)\|^2
			+\int_0^t\int_{\mathbb{R}}\left[\frac{\mu u_{xx}^2}{v}
			+\frac{\kappa\theta_{xx}^2}{v}\right]\lesssim 1.
		\end{equation}
		In view of \eqref{bound_v} and \eqref{apriori3},
		we can obtain \eqref{E_1order} from
		\eqref{E_basic}, \eqref{E_theta1}, \eqref{E_vx'}, and \eqref{E_1order1}.
		The proof of the lemma is finished.
	\end{proof}
	
	We present a
	local-in-time lower bound for the temperature $\theta(t,x)$
	in the following lemma.
	\begin{lemma} \label{L_lower}
		Assume that the conditions listed in Lemma \ref{L_v1} hold. Then
		there exist positive constant $C_4$
		depending
		only on $\Pi_0$,
		$V_0$, and
		$H(V_0)$
		such that
		\begin{equation} \label{lower_th}
			\inf_{\mathbb{R}}\theta(t,\cdot)\geq \frac{\inf_{\mathbb{R}}\theta(s,\cdot)}
			{C_4\inf_{\mathbb{R}}\theta(s,\cdot)(t-s)+1}
			\quad
			\textrm{for all }
			0\leq s\leq t\leq T.
		\end{equation}
	\end{lemma}
	\begin{proof}
		Multiply \eqref{id_theta} by $\theta^{-2}$ to have
		\begin{equation*}
			c_v\left(\frac{1}{\theta}\right)_t
			=\left[\frac{\kappa}{v}\left(\frac{1}{\theta}\right)_x\right]_x
			-\frac{2\theta\kappa}{v}\left|\left(\frac{1}{\theta}\right)_x\right|^2
			-\frac{\mu}{v\theta^2}\left[u_x-\frac{\theta}{2\mu}\right]^2
			+\frac{1}{4\mu v}.
		\end{equation*}
		In view of \eqref{apriori3} and \eqref{bound_v},
		we deduce that
		\begin{equation*}
			c_v\left(\frac{1}{\theta}\right)_t
			\leq \left[\frac{\kappa}{v}\left(\frac{1}{\theta}\right)_x\right]_x
			+C_4,
		\end{equation*}
		for some positive constant $C_4$, depending
		only on $\Pi_0$,
		$V_0$, and
		$H(V_0)$.
		
		Let $s\in[0,T]$ be fixed and define
		\begin{equation*}
			H(t,x):=\frac{1}{\theta(t,x)}-C_4(t-s).
		\end{equation*}
		Then we derive that $H$ satisfies
		\begin{equation*}
			\left\{
			\begin{aligned}
				&c_v H_t\leq \left(\frac{\kappa}{v}H_x\right)_x
				&& {\rm for }\ (t,x)\in(s,T]\times\mathbb{R},\\
				&H(s,x)=\frac{1}{\theta(s,x)}\leq \frac{1}{\inf_{\mathbb{R}}\theta(s,\cdot)}
				&& {\rm for }\ x\in\mathbb{R}.
			\end{aligned}\right.
		\end{equation*}
		Employing the maximum principle (see  \cite{E10MR2597943}),
		we infer that
		$$H(t,x)\leq \frac{1}{\inf_{\mathbb{R}}\theta(s,\cdot)} \quad
		{\rm for\ all}\ (t,x)\in[s,T]\times\mathbb{R},$$
		which implies \eqref{lower_th}. The proof is completed.
	\end{proof}
	
	\subsection{Estimates of second-order derivatives}\label{sec_2}
	In subsections
	\ref{sec_2}
	and \ref{sec_3}, to simplify the presentation,
	we introduce
	$A\lesssim_h B$ if $A\leq C_h B$
	holds uniformly for some constant $C_h$,
	depending
	only on
	$\Pi_0$,
	$V_0$, and
	$H(C_2)$
	with $C_2$ given in Lemma \ref{L_th2}.
	The letter $C(m_2)$ will be employed to denote
	some positive constant which depends only on
	$m_2$,
	$\Pi_0$,
	$V_0$, and
	$H(C_2)$.
	We note from \eqref{H} and \eqref{bound_v} that
	\begin{equation} \label{2.3H_est}
		\sup_{(t,x)\in[0,T]\times\mathbb{R}}\left|\left(h(v(t,x)), h'(v(t,x)), h''(v(t,x)),h'''(v(t,x))\right)\right|
		\leq H(C_2).
	\end{equation}

	We estimate the second-order derivatives of $(u(t,x),\theta(t,x))$  in the next lemma.
	\begin{lemma}
		\label{L_2order1}
		Assume that the conditions listed in Lemma \ref{L_v1} hold. Then
		\begin{align} \notag
			&\sup_{t\in[0,T]}\|(u_{xx},\theta_{xx})(t)\|^2
			+\int_0^T\|(u_{xxx},\theta_{xxx})(t)\|^2\mathrm{d}t\\
			&\qquad \label{E_2order1}
			\lesssim_h C(m_2)+\int_0^T\|v_{xx}(t)\|^2\mathrm{d}t
			+\sup_{t\in[0,T]}\|v_{xx}(t)\|^2.
		\end{align}
	\end{lemma}
	\begin{proof} The proof is divided into the following steps:
		
		\noindent   {\bf Step\,1}.
		Differentiating  $\eqref{NS_L}_2$ with repect to $x$,
		and multiplying the resulting identity by $u_{xxx}$ give
		\begin{equation*}
			\left[\frac12u_{xx}^2\right]_t
			-\left[u_{xt}u_{xx}\right]_x
			+\frac{\mu u_{xxx}^2}{v}
			=P_{xx}u_{xxx}
			+\left[\frac{\mu u_{xxx}}{v}-\left(\frac{\mu u_x}{v}\right)_{xx}\right]u_{xxx}.
		\end{equation*}
		Integrate the above identity over $[0,t]\times\mathbb{R}$,
		and use \eqref{bound_v}, \eqref{apriori3}, and Cauchy's inequality to obtain
		\begin{equation}\label{E_uxx1}
			\|u_{xx}(t)\|^2+\int_0^t\int_{\mathbb{R}}u_{xxx}^2
			\lesssim 1+
			\int_0^t\int_{\mathbb{R}}\left|P_{xx}\right|^2
			+\int_0^t\int_{\mathbb{R}}
			\left|\frac{\mu u_{xxx}}{v}-\left(\frac{\mu u_x}{v}\right)_{xx}\right|^2.
		\end{equation}
		We next make the estimates for the terms
		on the right-hand side of \eqref{E_uxx1}.
		In light of \eqref{bound_v},
		we deduce for general smooth function $f(v)$ that
		\begin{equation}\label{E_f}
			\left\{
			\begin{aligned}
				|f(v)_x|&\lesssim |f'(v)| |v_x|,\\
				|f(v)_{xx}|&\lesssim
				|(f',f'')(v)|(|v_{xx}|+v_x^2),\\
				|f(v)_{xxx}|&\lesssim
				|(f',f'',f''')(v)|(|v_{xxx}|+|v_{xx}|v_x^2 +|v_x|^3).
			\end{aligned} \right.
		\end{equation}
		Hence by using \eqref{upper_th} and
		\begin{equation} \label{id_Pxx}
			P_{xx}=\frac{\theta_{xx}}{v}+2\theta_x\left(\frac{ 1}{v}\right)_x
			+\theta \left(\frac{ 1}{v}\right)_{xx},
		\end{equation}
		we infer
		\begin{equation*}
			|P_{xx}|^2\lesssim
			|(v_{xx},\theta_{xx})|^2+|(v_{x},\theta_{x})|^4.
		\end{equation*}
		From \eqref{E_1order} and \eqref{apriori1}, we have
		\begin{equation*}
			\int_0^T\|v_x \|^2 \leq C(m_2),
		\end{equation*}
		which combined with \eqref{E_1order} implies
		\begin{align} \notag
			\int_0^t\int_{\mathbb{R}}
			|(v_x,u_x,\theta_x)|^4
			\lesssim&~\int_0^t\|(v_{xx},u_{xx},\theta_{xx})\|\|(v_x,u_x,\theta_x)\|^3
			\\ \notag
			\lesssim&~
			\sup_{ [0,T]}
			\|(v_x,u_x,\theta_x) \|^2
			\int_0^t\|(v_x,u_x,\theta_x)\|_1^2
			\\ \label{E_uxx1a}
			\lesssim&~C(m_2)+\int_0^t\|v_{xx} \|^2.
		\end{align}
		Consequently, we have
		\begin{equation}\label{E_uxx1b}
			\int_0^t\int_{\mathbb{R}}\left|P_{xx}\right|^2
			\lesssim
			\int_0^t\|v_{xx} \|^2+C(m_2).
		\end{equation}
		
		To estimate the last term in \eqref{E_uxx1},
		we first make some estimate of $\theta^{\alpha}$.
		It follows from \eqref{apriori3} that
		\begin{equation}\label{E_a}
			\left\{
			\begin{aligned}
				\left|\left(\theta^{\alpha}\right)_x\right|
				&\lesssim |\theta_x|,\\
				\left|\left(\theta^{\alpha}\right)_{xx}\right|
				&\lesssim |\theta_{xx}|+\theta_x^2,\\
				\left|\left(\theta^{\alpha}\right)_{xxx}\right|
				&\lesssim |\theta_{xxx}|+|\theta_{xx}|\theta_x^2 +|\theta_x|^3,
			\end{aligned} \right.
		\end{equation}
		which combined with \eqref{E_f} yields
		\begin{equation}\label{E_fa} \left\{
			\begin{aligned}
				|\left(f(v)\theta^{\alpha}\right)_x|
				\lesssim~&
				|(f,f')(v)| |(v_x,\theta_x)|,\\
				|\left(f(v)\theta^{\alpha}\right)_{xx}|
				\lesssim~&
				|(f,f',f'')(v)|\big[|(v_{xx},\theta_{xx})|+|(v_x,\theta_x)|^2\big],\\
				|\left(f(v)\theta^{\alpha}\right)_{xxx}|
				\lesssim~&
				|(f,f',f'',f''')(v)| \big[|(v_{xxx},\theta_{xxx})|\\
				&+|(v_{xx},\theta_{xx})|
				|(v_x,\theta_x)|^2 +|(v_x,\theta_x)|^3\big].
			\end{aligned} \right.
		\end{equation}
		Taking $f(v)=h(v)/v$,
		we can combine the identity
		\begin{equation*}
			\left(\frac{\mu u_x}{v}\right)_{xx}=
			\left(\frac{\mu }{v}\right)_{xx}u_x
			+2\left(\frac{\mu }{v}\right)_{x} u_{xx}
			+\frac{\mu }{v}u_{xxx}
		\end{equation*}
		and \eqref{2.3H_est}  to conclude
		\begin{equation*}
			\left|\left(\frac{\mu u_x}{v}\right)_{xx}-\frac{\mu u_{xxx}}{v}\right|
			\lesssim_h
			|(v_x,u_x,\theta_x)|^3
			+|v_{xx}||u_x|
			+|(u_{xx},\theta_{xx})||(v_x,u_x,\theta_x)|.
		\end{equation*}
		From this estimate, we derive
		\begin{align} \notag
			\int_0^t\int_{\mathbb{R}}\left|\frac{\mu u_{xxx}}{v}
			-\left(\frac{\mu u_x}{v}\right)_{xx}\right|^2
			\lesssim_h&
			\int_0^t\int_{\mathbb{R}}
			|(v_x,u_x,\theta_x)|^6
			+\int_0^t\int_{\mathbb{R}}
			v_{xx}^2u_x^2\\ \label{E_uxx1c}
			&
			+ \int_0^t\int_{\mathbb{R}}
			|(u_{xx},\theta_{xx})|^2|(v_x,u_x,\theta_x)|^2.
		\end{align}
		Employ Sobolev's inequality  and \eqref{E_1order} to get
		\begin{align} \notag
			\int_0^t\int_{\mathbb{R}}
			|(v_x,u_x,\theta_x)|^6
			\lesssim&~\int_0^t \|(v_x,u_x,\theta_x)\|_{L^\infty}^4
			\|(v_x,u_x,\theta_x)\|^2\\ \notag
			\lesssim&~\int_0^t\|(v_{xx},u_{xx},\theta_{xx})\|^2\|(v_x,u_x,\theta_x)\|^4
			\\ \label{E_uxx1d}
			\lesssim&~ 1+\int_0^t\|v_{xx}\|^2,
			\\ \label{E_uxx1e}
			\int_0^t\int_{\mathbb{R}}
			v_{xx}^2u_x^2
			\leq  &~ \sup_{ [0,t]}\|v_{xx} \|^2
			\int_0^t    \|u_x \|_1^2
			\leq   \sup_{ [0,t]}\|v_{xx} \|^2,
		\end{align}
		and 
		\begin{align} \notag
			\int_0^t\int_{\mathbb{R}}
			|(u_{xx},\theta_{xx})|^2|(v_x,u_x,\theta_x)|^2
			\lesssim~&
			\int_0^t\|(u_{xx},\theta_{xx}) \|_{L^\infty}^2
			\|(v_x,u_x,\theta_x) \|^2\\ \notag
			\lesssim~&
			\int_0^t \|(u_{xx},\theta_{xx}) \|
			\|(u_{xxx},\theta_{xxx}) \|\\ \label{E_uxx1f}
			\lesssim~& C(\delta)+\delta
			\int_0^t\|(u_{xxx},\theta_{xxx}) \|^2.
		\end{align}
		We plug \eqref{E_uxx1b} and \eqref{E_uxx1c} into \eqref{E_uxx1},
		and use \eqref{E_uxx1d}--\eqref{E_uxx1f} to have
		\begin{align} \notag
			&\|u_{xx}(t)\|^2+\int_0^t\int_{\mathbb{R}}u_{xxx}^2\\ 
			\label{E_uxx2} &\qquad
			\lesssim_h  C(m_2)+\int_0^t\|v_{xx} \|^2
			+\sup_{[0,t]}\|v_{xx} \|^2
			+\delta
			\int_0^t\|(u_{xxx},\theta_{xxx}) \|^2.
		\end{align}
		
		\noindent{\bf Step 2}. Next,
		we differentiate \eqref{id_theta} with repect to $x$
		and multiply the result by \,$\theta_{xxx}$ to find
		\begin{equation*}
			\begin{aligned}
				&\left[\frac{c_v}{2}\theta_{xx}^2\right]_t
				-\left[c_v\theta_{xt}\theta_{xx}\right]_x
				+\frac{\kappa\theta_{xxx}^2}{v}\\
				&\qquad =
				(Pu_x)_x\theta_{xxx}-\left(\frac{\mu u_x^2}{v}\right)_{x}\theta_{xxx}
				+\left[\frac{\kappa\theta_{xxx}}{v}
				-\left(\frac{\kappa\theta_x}{v}\right)_{xx}\right]\theta_{xxx}.
			\end{aligned}
		\end{equation*}
		Integrating this last identity over $[0,t]\times\mathbb{R}$,
		we obtain from Cauchy's inequality, \eqref{bound_v} and \eqref{apriori3} that
		\begin{align} \notag
			\|\theta_{xx}(t)\|^2+\int_0^t\int_{\mathbb{R}}\theta_{xxx}^2
			\lesssim ~&1+
			\int_0^t\int_{\mathbb{R}}\left|(Pu_x)_x\right|^2
			+\int_0^t\int_{\mathbb{R}}\left|\left(\frac{\mu u_x^2}{v}\right)_{x}
			\right|^2\\ \label{E_th_xx1}
			&+\int_0^t\int_{\mathbb{R}}\left|\frac{\kappa\theta_{xxx}}{v}
			-\left(\frac{\kappa\theta_x}{v}\right)_{xx}\right|^2.
		\end{align}
		We estimate the terms on the right-hand side of \eqref{E_th_xx1} below.
		First it follows from \eqref{upper_th} and \eqref{bound_v} that
		\begin{equation*}
			\begin{aligned}
				\left|(Pu_x)_x\right|
				\lesssim|P_xu_x|+|Pu_{xx}|
				\lesssim |(v_x,u_x,\theta_x)|^2+|u_{xx}|,
			\end{aligned}
		\end{equation*}
		which along with \eqref{E_1order} and \eqref{E_uxx1a} implies
		\begin{equation}\label{E_th_xx1a}
			\int_0^t\int_{\mathbb{R}}\left|(Pu_x)_x\right|^2
			\lesssim
			\int_0^t\int_{\mathbb{R}}\left[
			|(v_x,u_x,\theta_x)|^4+u_{xx}^2\right]
			\lesssim
			C( m_2)+ \int_0^t\|v_{xx} \|^2.
		\end{equation}
		We deduce from \eqref{E_fa} with $f(v)=h(v)/v$ that
		\begin{equation*}
			\left|\left(\frac{\mu u_x^2}{v}\right)_x\right|
			=\left|\left(\frac{\mu}{v}\right)_x u_x^2+\frac{2\mu u_xu_{xx}}{v}\right|
			\lesssim_h |(v_x,u_x,\theta_x)|^3+|u_x u_{xx}|.
		\end{equation*}
		In light of \eqref{E_uxx1d} and \eqref{E_uxx1f}, we have
		\begin{align} \notag
			\int_0^t\int_{\mathbb{R}} \left|\left(\frac{\mu u_x^2}{v}\right)_x\right|^2
			\lesssim_h\,& \int_0^t\int_{\mathbb{R}}
			\left[|(v_x,u_x,\theta_x)|^6+u_x ^2u_{xx}^2\right]\\ \label{E_th_xx1b}
			\lesssim_h\,&
			C(\delta)+\int_0^t\|v_{xx}\|^2
			+\delta\int_0^t\|(u_{xxx},\theta_{xxx}) \|^2.
		\end{align}
		For the last term in \eqref{E_th_xx1},
		we deduce by applying the argument in Step 1 that
		\begin{align} \notag
			&\int_0^t\int_{\mathbb{R}}\left|\frac{\kappa\theta_{xxx}}{v}
			-\left(\frac{\kappa\theta_x}{v}\right)_{xx}\right|^2\\
			&\qquad
			\lesssim_h
			C( \delta)+ \int_0^t\|v_{xx}\|^2
			+ \sup_{ [0,t]}\|v_{xx} \|^2
			+ \delta
			\int_0^t\|(u_{xxx},\theta_{xxx} )\|^2. \label{E_th_xx1c}
		\end{align}
		Plug \eqref{E_th_xx1a}-\eqref{E_th_xx1c} into \eqref{E_th_xx1} to get
		\begin{align} \notag
			&\|\theta_{xx}(t)\|^2+\int_0^t \int_{\mathbb{R}}\theta_{xxx}^2\\
			&\qquad
			\lesssim_h C(m_2)+\int_0^t\|v_{xx} \|^2
			+\sup_{ [0,t]}\|v_{xx} \|^2
			+\delta
			\int_0^t\|(u_{xxx},\theta_{xxx}) \|^2. \label{E_th_xx2}
		\end{align}
		Combining \eqref{E_uxx2} and \eqref{E_th_xx2},
		we take $\delta$ small enough to prove \eqref{E_2order1}.
		This completes the proof.
	\end{proof}
	We next obtain a $m_2$-dependent bound for the
	second-order derivatives with respect to $x$ of the solution $(v(t,x),$ $u(t,x), \theta(t,x))$.
	\begin{lemma}
		\label{L_2order}
		Assume that the conditions listed in Lemma \ref{L_v1} hold. Then
		\begin{equation}\label{E_2order}
			\sup_{t\in[0,T]}\|(v_{xx},u_{xx},\theta_{xx})(t)\|^2
			+\int_0^T\|(v_{xx},u_{xxx},\theta_{xxx})(t)\|^2\mathrm{d}t
			\leq C(m_2).
		\end{equation}
	\end{lemma}
	\begin{proof}
		Differentiate \eqref{id1} with respect to $x$
		and multiply the result by $\left(\frac{\mu v_x}{v}\right)_x$ to find
		\begin{equation*}
			\begin{aligned}
				&\left[\frac12\left(\frac{\mu v_x}{v}\right)_x^2\right]_t
				-\left[u_x\left(\frac{\mu v_x}{v}\right)_x\right]_t
				+\left[u_x\left(\frac{\mu v_x}{v}\right)_t\right]_x\\
				&\qquad = u_{xx}\left(\frac{\mu v_x}{v}\right)_t
				+\left(\frac{\theta}{v}\right)
				_{xx}\left(\frac{\mu v_x}{v}\right)_x
				+\left(\frac{\mu v_x}{v}\right)_x
				\left[\frac{\mu_\theta}{v}(v_x\theta_t-\theta_xu_x)\right]_x.
			\end{aligned}
		\end{equation*}
		We integrate the above identity over $[0,t]\times\mathbb{R}$
		and use Cauchy's inequality to derive
		\begin{align} \notag
			\left\|\left(\frac{\mu v_x}{v}\right)_x(t)\right\|^2
			\lesssim_h\,& 1
			+\int_0^t\int_{\mathbb{R}}
			u_{xx}\left(\frac{\mu v_x}{v}\right)_t
			+\int_0^t\int_{\mathbb{R}}\left(\frac{\theta}{v}\right)
			_{xx}\left(\frac{\mu v_x}{v}\right)_x\\
			&+\int_0^t\int_{\mathbb{R}}\left(\frac{\mu v_x}{v}\right)_x
			\left[\frac{\mu_\theta}{v}(v_x\theta_t-\theta_xu_x)\right]_x
			. \label{E_vxx1}
		\end{align}
		It follows from \eqref{pro2.1}, \eqref{apriori3}, and
		\begin{equation*}
			\begin{aligned}
				\left(\frac{\mu v_x}{v}\right)_t
				=\frac{v\mu_v-\mu}{v^2}v_xu_x+\frac{\mu_\theta\theta_tv_x}{v}
				+\frac{\mu u_{xx}}{v}
			\end{aligned}
		\end{equation*}
		that
		\begin{equation*}
			\begin{aligned}
				\left|\left(\frac{\mu v_x}{v}\right)_t\right|
				\lesssim_h |v_xu_x|+N m_2^{-1}|\alpha||v_x|
				+| u_{xx}|.
			\end{aligned}
		\end{equation*}
		We then deduce from Cauchy's inequality and \eqref{E_1order} that
		\begin{equation}\label{E_vxx1a}
			\begin{aligned}
				\int_0^t\int_{\mathbb{R}}u_{xx}\left(\frac{\mu v_x}{v}\right)_t
				\lesssim C(m_2).
			\end{aligned}
		\end{equation}
		In view of  \eqref{id_Pxx} and
		\begin{equation}\label{id3}
			\left(\frac{\mu v_x}{v}\right)_x
			=\frac{\mu }{v}v_{xx}
			+\left(\frac{\mu }{v}\right)_xv_x,
		\end{equation}
		we have
		\begin{align} \notag
			&	\int_0^t\int_{\mathbb{R}}\left(\frac{\theta}{v}\right)
			_{xx}\left(\frac{\mu v_x}{v}\right)_x\\
			&\qquad	\leq-\int_0^t\int_{\mathbb{R}}\frac{\theta\mu v_{xx}^2}{2v^3}
			+C(m_2)\int_0^t\int_{\mathbb{R}}
			\left[\theta_{xx}^2+\left|(\theta_x,v_x)\right|^4\right]. \label{E_vxx1b}
		\end{align}
		Apply Sobolev's inequality to get
		\begin{equation*}
			\begin{split}
				\int_0^t\int_{\mathbb{R}}\left|(\theta_x,v_x)\right|^4
				\lesssim~&
				\int_0^t
				\left\|(\theta_x,v_x) \right\|^2\left\|(\theta_x,v_x) \right\|_1^2\\
				\lesssim~& C(m_2)+
				\int_0^t
				\left\|(\theta_x,v_x) \right\|^2\left\| v_{xx} \right\|^2.
			\end{split}
		\end{equation*}
		Inserting the last inequality and
		\eqref{E_1order} into \eqref{E_vxx1b}, we infer
		\begin{align} \notag
			&\int_0^t\int_{\mathbb{R}}\left(\frac{\theta}{v}\right)_{xx}
			\left(\frac{\mu v_x}{v}\right)_x\\
			&\qquad
			\leq-\int_0^t\int_{\mathbb{R}}\frac{\theta\mu v_{xx}^2}{2v^3}
			+C(m_2)
			+C(m_2)\int_0^t
			\left\|(\theta_x,v_x) \right\|^2\left\| v_{xx} \right\|^2. \label{E_vxx1c}
		\end{align}
		
		For the last term in \eqref{E_vxx1},
		we use \eqref{id3},
		\begin{equation*}
			\begin{aligned}
				\left[\frac{\mu_\theta}{v}(v_x\theta_t-\theta_xu_x)\right]_x
				=~&\frac{\mu_\theta}{v}\left(v_{xx}\theta_t+v_x\theta_{xt}
				-\theta_{xx}u_x-\theta_xu_{xx}\right)\\
				&+\left[\frac{\mu_{\theta\theta}\theta_x
					+\mu_{\theta v}v_x}{v}-\frac{\mu_{\theta}v_x}{v^2}\right]
				\left(v_x\theta_t-\theta_xu_x\right),
			\end{aligned}
		\end{equation*}
		and \eqref{apriori3} to derive
		\begin{align*}
			\left|\left(\frac{\mu v_x}{v}\right)_x\right|
			\lesssim_h\,&|v_{xx}|+|(v_x,\theta_x)|^2,
			\\
			\left|\left[\frac{\mu_\theta}{v}(v_x\theta_t-\theta_xu_x)\right]_x\right|
			\lesssim_h\,&
			|\alpha|(|\theta_t||v_{xx}|+|\theta_{xt}||v_x|
			+|\theta_{xx}||u_x|+|\theta_x||u_{xx}|)
			\\
			&+|\alpha||(v_x,\theta_x)|(|v_x\theta_t|+|\theta_xu_x|).
		\end{align*}
		It follows from the identity
		\begin{equation*}
			c_v\theta_{tx}=\left(\frac{\mu u_x^2}{v}\right)_x
			+\left(\frac{\kappa\theta_x}{v}\right)_{xx}
			-\left(\frac{\theta}{v}u_x\right)_x
		\end{equation*}
		and \eqref{upper_th}--\eqref{bound_v} that
		\begin{equation*}
			\begin{aligned}
				|\theta_{tx}|\leq~&
				|\theta_{xxx}|+|(v_x,\theta_x)||\theta_{xx}|+(1+|u_x|)|u_{xx}|\\
				&+|\theta_x||v_{xx}|+(1+|(u_x,\theta_x)|)|(v_x,u_x,\theta_x)|^2.
			\end{aligned}
		\end{equation*}
		Hence applying Cauchy's inequality yields
		\begin{align} \notag
			&\int_0^t\int_{\mathbb{R}}\left|\left(\frac{\mu v_x}{v}\right)_x
			\left[\frac{\mu_\theta}{v}(v_x\theta_t-\theta_xu_x)\right]_x\right|\\ \notag
			\lesssim_h\,&
			\left(\epsilon+C(\epsilon)\alpha^2\lrn(\theta_t,\theta_x v_x)\rrn^2
			\right)\int_0^t\int_{\mathbb{R}}v_{xx}^2
			+C(\epsilon)\alpha^2\lrn v_x\rrn^2\int_0^t\int_{\mathbb{R}}\theta_{xxx}^2
			\\ \notag
			&+C(\epsilon)\int_0^t\int_{\mathbb{R}} \alpha^2|(v_x,\theta_x)|^2
			|(v_x\theta_t,\theta_xu_x)|^2 
			+C(\epsilon)\int_0^t\int_{\mathbb{R}}|(v_x,\theta_x)|^4  \\ \notag
			&+C(\epsilon)\alpha^2\int_0^t\int_{\mathbb{R}}
			\left[|(v_x,u_x,\theta_x)|^4\theta_{xx}^2
			+(1+|(u_x,\theta_x)|^2)|(v_x,u_x,\theta_x)|^6\right]\\ \notag
			&+C(\epsilon) \alpha^2\int_0^t\int_{\mathbb{R}}
			\left[(1+|u_x|^2)u_{xx}^2v_x^2+\theta_x^2u_{xx}^2\right]\\ \label{E_vxx1d}
			\lesssim_h\,&\epsilon\int_0^t\int_{\mathbb{R}}v_{xx}^2
			+C(\epsilon,m_2).
		\end{align}
		Here we have use \eqref{apriori3} and
		\begin{equation*}
			\int_0^t\int_{\mathbb{R}}\theta_{xxx}^2\lesssim N^2.
		\end{equation*}
		Plug \eqref{E_vxx1a}, \eqref{E_vxx1c}, and \eqref{E_vxx1d}
		into \eqref{E_vxx1} to deduce
		\begin{equation*}
			\begin{aligned}
				\|v_{xx}(t)\|^2+\int_0^t \|v_{xx} \|^2
				\lesssim_h C(m_2)+
				\int_0^t
				\left\|(\theta_x,v_x) \right\|^2\left\| v_{xx} \right\|^2.
			\end{aligned}
		\end{equation*}
		We apply Gronwall's inequality to the above estimate to obtain
		\begin{equation*}
			\begin{aligned}
				\|v_{xx}(t)\|^2+\int_0^t \|v_{xx} \|^2
				\lesssim C(m_2),
			\end{aligned}
		\end{equation*}
		which combined with \eqref{E_2order1} implies \eqref{E_2order}.
		The proof is completed.
	\end{proof}

	\subsection{Estimates of third-order derivatives}\label{sec_3}
	Estimates on the third-order derivatives of $(v(t,x), u(t,x), \theta(t,x))$ with respect to $x$ will be proved in this subsection.
	The notation
	$A\lesssim_h B$ is employed to denote that
	$A\leq C_h B$  holds uniformly for some constant $C_h$,
	depending only on $\Pi_0$,  $V_0$, and $H(C_2)$
	with $C_2$ given in Lemma \ref{L_th2}.
	And we denote by $C(m_2)$
	some positive constant which depends only on
	$m_2$,  $\Pi_0$,  $V_0$, and $H(C_2)$.
	
	We first give an estimate on the third-order derivatives of $u$ and $\theta$.
	\begin{lemma}\label{L_3order1}
		Assume that the conditions listed in Lemma \ref{L_v1} hold. Then
		\begin{align}\notag
			&\sup_{t\in[0,T]}\|(u_{xxx},\theta_{xxx})(t)\|^2
			+\int_0^T\|(u_{xxxx},\theta_{xxxx})(t)\|^2\mathrm{d}t\\
			&\qquad
			\lesssim_h C(m_2)+\int_0^T\|v_{xxx}(t)\|^2\mathrm{d}t
			+\sup_{t\in[0,T]}\|v_{xxx}(t)\|^2. \label{E_3order1}
		\end{align}
	\end{lemma}
	\begin{proof} The proof is divided into the following steps:
		
		\noindent{\bf Step\,1.}
		Differentiating $\eqref{NS_L}_2$ with repect to $x$ twice
		and multiplying the resulting identity by $u_{xxxx}$ yield
		\begin{equation*}
			\begin{aligned}
				&\frac12\left(u_{xxx}^2\right)_t
				-\left(u_{xxt}u_{xxx}\right)_x
				+\frac{\mu u_{xxxx}^2}{v}\\
				&\qquad
				=P_{xxx}u_{xxxx}+\left[\frac{\mu u_{xxxx}}{v}
				-\left(\frac{\mu u_x}{v}\right)_{xxx}\right]u_{xxxx}.
			\end{aligned}
		\end{equation*}
		Integrate the above identity over $[0,t]\times\mathbb{R}$ to have
		\begin{align}\notag
			&\|u_{xxx}(t)\|^2+\int_0^t\int_{\mathbb{R}}u_{xxxx}^2\\ \label{E_uxxx1}
			&\qquad	\lesssim  1+
			\int_0^t\int_{\mathbb{R}}\left|P_{xxx}\right|^2
			+\int_0^t\int_{\mathbb{R}}
			\left|\frac{\mu u_{xxxx}}{v}-\left(\frac{\mu u_x}{v}\right)_{xxx}\right|^2.
		\end{align}
		We compute from \eqref{id_Pxx}  that
		\begin{equation*}
			P_{xxx}=\frac{\theta_{xxx}}{v}-3\frac{\theta_{xx} v_x}{v^2}
			-3\frac{\theta_x v_{xx}}{v^2}+6\frac{\theta_x v_x^2}{v^3}
			-\frac{\theta v_{xxx}}{v^2}+6\frac{\theta v_xv_{xx}}{v^3}
			-6\frac{\theta v_x^3}{v^4}.
		\end{equation*}
		Hence
		\begin{equation*}
			|P_{xxx}|^2\lesssim |(v_{xxx},\theta_{xxx})|^2
			+|(v_x,\theta_x)|^6+|(v_x,\theta_x)|^2|(v_{xx},\theta_{xx})|^2.
		\end{equation*}
		It follows from \eqref{E_uxx1d}--\eqref{E_uxx1f} and \eqref{E_2order}\,that
		\begin{equation}\label{E_uxxx1a}
			\int_0^t\int_{\mathbb{R}}\left|P_{xxx}\right|^2
			\lesssim
			C(m_2) +\int_0^t\int_{\mathbb{R}}v_{xxx}^2.
		\end{equation}
		From  \eqref{E_fa} with $f(v )=\frac{h(v)}{v}$ and
		\begin{equation*}
			\begin{aligned}
				\left(\frac{\mu u_x}{v}\right)_{xxx}
				=\left(\frac{\mu }{v}\right)_{xxx}u_x
				+3\left(\frac{\mu }{v}\right)_{xx }u_{xx}
				+3\left(\frac{\mu }{v}\right)_{x }u_{xxx}
				+\frac{\mu }{v} u_{xxxx},
			\end{aligned}
		\end{equation*}
		we have
		\begin{equation*}
			\begin{aligned}
				&\left|\left(\frac{\mu u_x}{v}\right)_{xxx}-\frac{\mu u_{xxxx}}{v}\right|
				\lesssim_h |(v_x,u_x,\theta_x)|^4
				+|(v_x,u_x,\theta_x)|^2|(v_{xx},u_{xx},\theta_{xx})|
				\\       &\qquad
				+|u_x||v_{xxx}|
				+|(v_x,u_x,\theta_x)||(u_{xxx},\theta_{xxx})|+|u_{xx}||(v_{xx},\theta_{xx})|.
			\end{aligned}
		\end{equation*}
		In view of \eqref{E_1order} and \eqref{E_2order}, we deduce
		\begin{align} \notag 
			&\int_0^t\int_{\mathbb{R}}\left|\frac{\mu u_{xxxx}}{v}
			-\left(\frac{\mu u_x}{v}\right)_{xxx}\right|^2
			\\       &\qquad
			\lesssim_h   C(\delta,m_2)
			+\sup_{ [0,t]}\|v_{xxx} \|^2
			+\delta
			\int_0^t\|(u_{xxxx},\theta_{xxxx}) \|^2. \label{E_uxxx1b}
		\end{align}
		Plugging \eqref{E_uxxx1a} and \eqref{E_uxxx1b} into \eqref{E_uxxx1},
		we get
		\begin{align} \notag 
			&\|u_{xxx}(t)\|^2+\int_0^t\int_{\mathbb{R}} u_{xxxx}^2
			\\       &\qquad
			\lesssim_h C(\delta,m_2) +\int_0^t\int_{\mathbb{R}}v_{xxx}^2
			+\sup_{ [0,t]}\|v_{xxx} \|^2
			+\delta
			\int_0^t\|(u_{xxxx},\theta_{xxxx}) \|^2. \label{E_uxxx2}
		\end{align}
		
		\noindent{\bf Step\,2.}
		We differentiate $\eqref{id_theta}$ with repect to $x$ twice
		and multiply the resulting identity by $\theta_{xxxx}$ to obtain
		\begin{equation*}
			\begin{aligned}
				&\frac{c_v}{2}\left(\theta_{xxx}\right)_t
				-c_v\left(\theta_{xxt}\theta_{xxx}\right)_x
				+\kappa\frac{\theta_{xxxx}^2}{v}      \\       &\qquad
				= (Pu_x)_{xx}\theta_{xxxx}+\left(\frac{\mu u_x^2}{v}\right)_{xx}\theta_{xxxx}
				+\left[\frac{\kappa\theta_{xxxx}}{v}
				-\left(\frac{\kappa\theta_x}{v}\right)_{xxx}\right]\theta_{xxxx}.
			\end{aligned}
		\end{equation*}
		Integrate the above identity over $[0,t]\times\mathbb{R}$ to have
		\begin{align} \notag 
			\|\theta_{xxx}(t)\|^2+\int_0^t\int_{\mathbb{R}}\theta_{xxxx}^2         
			\lesssim~&
			\int_0^t\int_{\mathbb{R}}\left|(Pu_x)_{xx}\right|^2
			+\int_0^t\int_{\mathbb{R}}\left|\left(\frac{\mu u_x^2}{v}\right)_{xx}\right|^2\\ \label{E_th_xxx1}
			&+\int_0^t\int_{\mathbb{R}}\left|\frac{\kappa\theta_{xxxx}}{v}
			-\left(\frac{\kappa\theta_x}{v}\right)_{xxx}\right|^2.
		\end{align}
		Similar to the derivation of \eqref{E_uxxx1b}, we can obtain
		\begin{align} \notag 
			&\int_0^t\int_{\mathbb{R}}\left|\left(\frac{\kappa\theta_x}{v}\right)_{xxx}
			-\frac{\kappa\theta_{xxxx}}{v}\right|^2      \\       &\qquad
			\lesssim_h   C(\delta,m_2)
			+\sup_{[0,t]}\|v_{xxx}\|^2
			+\delta
			\int_0^t\|(u_{xxxx},\theta_{xxxx})\|^2. \label{E_th_xxx1a}
		\end{align}
		The identity
		\begin{equation*}
			(Pu_x)_{xx}=(Pu_{xx}+P_xu_x)_x
			=Pu_{xxx}+2P_xu_{xx}+P_{xx}u_x
		\end{equation*}
		implies
		\begin{equation*}
			\begin{aligned}
				\left|(Pu_x)_{xx}\right|^2\lesssim
				u_{xxx}^2+\theta_x^2u_{xx}^2+v_x^2u_{xx}^2+u_x^2v_{xx}^2
				+u_x^2\theta_{xx}^2+v_x^2u_x^2\theta_x^2+v_x^4u_x^2.
			\end{aligned}
		\end{equation*}
		Hence
		\begin{equation}\label{E_th_xxx1b}
			\int_0^t\int_{\mathbb{R}}\left|(Pu_x)_{xx}\right|^2
			\lesssim C(m_2).
		\end{equation}
		On the other hand, from
		\begin{equation*}
			\left(\frac{\mu u_x^2}{v}\right)_{x}
			=\frac{\mu_x u_x^2}{v}+\frac{2\mu u_xu_{xx}}{v}-\frac{\mu u_x^2v_x}{v^2},
		\end{equation*}
		and
		\begin{equation*}
			\begin{aligned}
				\left(\frac{\mu u_x^2}{v}\right)_{xx}
				=&~\frac{\mu_{xx} u_x^2}{v}+\frac{4\mu_x u_xu_{xx}}{v}
				+\frac{2\mu u_{xx}^2}{v}+\frac{2\mu u_xu_{xxx}}{v}\\
				&~-\frac{2\mu_x u_x^2v_x}{v^2}-\frac{4\mu u_xu_{xx}v_x}{v^2}
				-\frac{\mu u_x^2v_{xx}}{v^2}+\frac{2\mu u_x^2v_x^2}{v^3},
			\end{aligned}
		\end{equation*}
		we have
		\begin{equation*}
			\begin{aligned}
				\left|\left(\frac{\mu u_x^2}{v}\right)_{xx}\right|
				\lesssim_h\,&
				|(v_x,u_x,\theta_x)|^4 +u_{xx}^2+|u_xu_{xxx}|\\
				&+|(v_x,u_x,\theta_x)|^2|(v_{xx},u_{xx},\theta_{xx})|.
			\end{aligned}
		\end{equation*}
		Thus,
		\begin{equation}\label{E_th_xxx1c}
			\int_0^t\int_{\mathbb{R}}\left|\left(\frac{\mu u_x^2}{v}\right)_{xx}\right|^2
			\lesssim C(m_2).
		\end{equation}
		Plug \eqref{E_th_xxx1a}--\eqref{E_th_xxx1c} into \eqref{E_th_xxx1} to deduce
		\begin{align} \notag 
			&\|\theta_{xxx}(t)\|^2+\int_0^t\int_{\mathbb{R}} \theta_{xxxx}^2       \\       &\qquad
			\lesssim_h C(\delta,m_2) +\int_0^t\int_{\mathbb{R}}v_{xxx}^2
			+\sup_{ [0,t]}\|v_{xxx} \|^2
			+\delta
			\int_0^t\|(u_{xxxx},\theta_{xxxx}) \|^2. \label{E_th_xxx2}
		\end{align}
		Combining \eqref{E_uxxx2} and \eqref{E_th_xxx2},
		we take $\delta$ suitably small to derive \eqref{E_3order1}.
	\end{proof}
	By using
	\eqref{apriori3} and Gronwall's inequality,
	we can deduce the $m_2$-dependent bounds for
	the third-order derivatives of $(v(t,x), u(t,x), \theta(t,x))$.
	The proof is
	similar to that of Lemma \ref{L_2order} and hence
	we omit the details for brevity.
	\begin{lemma}
		\label{L_3order}
		Assume that the conditions listed in Lemma \ref{L_v1} hold. Then 
		for all $t\in[0,T]$, we have
		\begin{equation}\label{E_3order2}
			\begin{aligned}
				\|(v_{xxx},u_{xxx},\theta_{xxx})(t)\|^2
				+\int_0^T\|(v_{xxx},u_{xxxx},\theta_{xxxx})(s)\|^2\mathrm{d}s
				\leq C(m_2).
			\end{aligned}
		\end{equation}
	\end{lemma}
	
	By virtue of  Lemmas \ref{L_bas}--\ref{L_3order},
	we can get the following corollary.
	\begin{corollary} \label{cor1}
		Assume that the conditions listed in Lemma \ref{L_v1} hold. Then
		there exists  $C(m_2)>0$,
		which depends only on
		$m_2$, $\Pi_0$,
		$V_0$, and
		$H(C_2)$
		with $C_2$ being given in Lemma \ref{L_th2},
		such that for all $t\in[0,T]$,
		\begin{equation}\label{E}
			\begin{aligned}
				\|(v-1,u ,\theta-1)(t)\|_3^2
				+\int_0^T
				\left[\|v_{x}(s)\|_2^2+\|(u_{x},\theta_{x})(s)\|_3^2\right]\mathrm{d}s
				\leq C(m_2).
			\end{aligned}
		\end{equation}
	\end{corollary}

	\section{Proof of Theorm \ref{thm}}\label{sec_proof}
	
	In this section we will prove our main result, Theorem \ref{thm}.
	For this purpose,
	we first present the local solvability result
	to the Cauchy problem \eqref{NS_L}--\eqref{far}, \eqref{transport}--\eqref{h}
	in the following lemma,
	which can be proved by the standard iteration method (see  \cite{MN04book}).
	\begin{lemma} \label{local}
		If  positive constants $M$ and $\lambda_i$ ($i=1,2$)
		exist such that
		$\|(v_0-1,u_0,\theta_0-1)\|_3\leq M$,
		$ v_0(x)\geq \lambda_1$,
		and $\theta_0(x)\geq \lambda_2$
		for all $x\in\mathbb{R}$, then
		there exists $T_0=T_0(\lambda_1,\lambda_2,M)>0$,
		depending only on
		$\lambda_1$, $\lambda_2$ and $M$, such that
		the Cauchy problem \eqref{NS_L}--\eqref{far}, \eqref{transport}--\eqref{h}
		has a unique
		solution $(v, u, \theta)
		\in X(0,T_0;\frac{1}{2}\lambda_1,
		\frac{1}{2}\lambda_2,2M)$.
	\end{lemma}
	
	We prove Theorem \ref{thm} in the following six steps
	by employing the continuation argument.
	
	\vspace*{2mm}
	
	\noindent{\bf Step\,1}.
	Set $T_1=128 C_3^4$, where $C_3$ is exactly the same constant as in \eqref{E_1order}.
	Recalling \eqref{thm1a}--\eqref{thm1b} and
	applying Lemma \ref{local},  we can find a positive constant
	$t_1= \min\{T_1,T_0(V_0,V_0,\Pi_0)\}$
	such that there exists
	a unique solution  $(v, u, \theta)\in X(0,t_1;\frac{1}{2}V_0,
	\frac{1}{2}V_0,2\Pi_0)$ to
	the Cauchy problem \eqref{NS_L}--\eqref{far}, \eqref{transport}--\eqref{h}.
	
	Take $|\alpha|\leq \alpha_1$, where $\alpha_1$ is some positive constant such that
	\begin{equation}
		\label{alpha1}
		\left(\tfrac{1}{2}V_0\right)^{-\alpha_1}\leq 2,\quad
		\left(2\Pi_0\right)^{\alpha_1}\leq 2,\quad
		\Xi\left(\tfrac{1}{2}V_0,
		\tfrac{1}{2}V_0,2\Pi_0\right)\alpha_1\leq \epsilon_1,
	\end{equation}
	where the value of $\epsilon_1$ is chosen in Lemma \ref{L_v1}.
	Then we can apply Lemmas \ref{L_th2} and \ref{L_lower}
	with $T=t_1$ to deduce that for each $t\in[0,t_1]$,
	the local solution
	$(v, u, \theta)$ constructed above satisfies
	\begin{gather}\label{N3.2}
		\theta(t,x)\geq \frac{V_0}
		{C_4 V_0T_1+1}=:C_5 \quad {\rm for\ all}\ x\in \mathbb{R},\\
		\label{N3.3}
		\theta(t,x)\leq C_1,
		\quad
		C_2\leq v(t,x)\leq C_2^{-1}\quad {\rm for\ all}\ x\in \mathbb{R},\\
		\label{N3.3a}
		\|(v-1,u,\theta-1)(t)\|_1^2
		+\int_0^{t}
		\left[
		\left\|\sqrt\theta v_x(s)\right\|^2+\|(u_{x},\theta_{x})(s)\|_1^2\right]\mathrm{d}s
		\leq C_3^2.
	\end{gather}
	Combining Corollary \ref{cor1} and \eqref{N3.2},
	we can find a  positive constant $C_6$,
	which depends on
	$C_5$, $\Pi_0$,
	$V_0$, and
	$H(C_2)$,
	such that for each $t\in[0,t_1]$,
	\begin{equation} \label{N3.4}
		\|(v-1,u ,\theta-1)(t)\|_3^2
		+\int_0^{t}      \left[\|v_{x}(s)\|_2^2+\|(u_{x},\theta_{x})(s)\|_3^2\right]\mathrm{d}s
		\leq C_6^2.
	\end{equation}
	
	\vspace*{2mm}
	
	\noindent{\bf Step\,2}.
	If we take  $(v(t_1,\cdot), u(t_1,\cdot), \theta(t_1,\cdot))$ as the initial data and apply
	Lemma \ref{local} again, we can extend the local solution  $(v,u,\theta)$
	to the time interval $[0,t_1+t_2]$ with
	$$t_2=\min\left\{T_1-t_1, T_0(  C_2 ,C_5,C_6)\right\}.$$
	Moreover,  for all $(t,x)\in[t_1,t_1+t_2]\times\mathbb{R} $,
	we have
	\begin{equation*}
		v(t,x)\geq \tfrac{1}{2}C_2,\quad
		\theta(t,x)\geq  \tfrac{1}{2}C_5,
	\end{equation*}
	and
	\begin{equation}\label{N3.5}
		\|(v-1,u ,\theta-1)(t)\|_3^2
		+\int_{t_1}^t
		\left[\|v_{x}(s)\|_2^2+\|(u_{x},\theta_{x})(s)\|_3^2\right]\mathrm{d}s
		\leq 4C_6^2,
	\end{equation}
	which combined with \eqref{N3.4} implies  that for all $t\in[0,t_1+t_2]$,
	\begin{equation}\label{N3.6}
		\|(v-1,u ,\theta-1)(t)\|_3^2
		+\int_{0}^t
		\left[\|v_{x}(s)\|_2^2+\|(u_{x},\theta_{x})(s)\|_3^2\right]\mathrm{d}s
		\leq 5C_6^2.
	\end{equation}
	
	Take $|\alpha|\leq \min\{\alpha_1,\alpha_2\}$,
	where $\alpha_1>0$ is determined
	by \eqref{alpha1} and
	$\alpha_2$ is some positive constant satisfying
	\begin{equation}
		\label{alpha2}
		\left(\tfrac{1}{2}C_5\right)^{-\alpha_2}\leq 2,\quad
		\left(2\sqrt{5}C_6\right)^{\alpha_2}\leq 2,\quad
		\Xi\left(\tfrac{1}{2}C_2,
		\tfrac{1}{2}C_5,\sqrt{5}C_6\right)\alpha_2\leq \epsilon_1,
	\end{equation}
	where the value of $\epsilon_1$ is chosen in Lemma \ref{L_v1}.
	Then we can employ Lemma \ref{L_th2}, Lemma \ref{L_lower},
	and Corollary \ref{cor1} with $T=t_1+t_2$
	to infer that
	the local solution $(v, u,\theta)$
	satisfies  \eqref{N3.2}--\eqref{N3.4}
	for each $t\in[0,t_1+t_2]$.
	
	\vspace*{2mm}
	
	\noindent{\bf Step\,3}.
	We repeat the argument in Step 2, to extend our
	solution $(v, u,\theta)$  to the time interval $[0,t_1+t_2+t_3]$,
	where
	$$t_3=\min\left\{T_1-(t_1+t_2), T_0( C_2 ^{-1},
	C_5,C_6)\right\}.$$
	Assume that $|\alpha|\leq \min\{\alpha_1,\alpha_2\}$
	with constants $\alpha_1$ and $\alpha_2$ satisfying
	\eqref{alpha1} and \eqref{alpha2}.
	Continuing, after finitely many steps we construct the unique solution
	$(v, u,\theta)$ existing on $[0,T_1]$ and satisfying
	\eqref{N3.2}--\eqref{N3.4}  for each $t\in[0,T_1]$.
	
	\vspace*{2mm}
	
	\noindent{\bf Step\,4}.
	Since $T_1= 128 C_3^4$ and
	\begin{equation}\label{N3.1}
		\sup_{0\leq t\leq T_1}\|(\theta-1)(t)\|_1^2
		+\int_{{T_1}/{2}}^{T_1}\|\theta_{x}(t)\|_1^2\mathrm{d}t
		\leq C_3^2,
	\end{equation}
	we can find a $t_0'\in[T_1/2,T_1]$ such that
	$$\|\theta(t_0')-1\|\leq C_3,\quad
	\|\theta_x(t_0')\|\leq \tfrac{1}{8}C_3^{-1}.$$
	For if not, we have
	that $\|\theta_x(t)\|> \frac{1}{8}C_3^{-1}$ for each $t \in[T_1/2,T_1]$
	and hence
	\begin{equation*}
		\int_{{T_1}/{2}}^{T_1}\|\theta_{x}(t)\|_1^2\mathrm{d}t
		>\tfrac{1}{2}T_1 \left(\tfrac{1}{8}C_3^{-1}\right)^2
		= C_3^2.
	\end{equation*}
	This contradicts \eqref{N3.1}.
	Then it follows from
	Sobolev's inequality that
	$$\left\|(\theta-1)(t_0')\right\|_{L^{\infty}}\leq
	\sqrt2\left\|(\theta-1)(t_0')\right\|^{\frac12}\left\|\theta_x(t_0')\right\|^{\frac12}
	\leq \tfrac{1}{2},$$
	from which we get
	\begin{equation} \label{3theta1}
		\theta(t_0',x)\geq 1-\left\|(\theta-1)(t_0')\right\|_{L^{\infty}}
		\geq \tfrac{1}{2}\quad {\rm for\ all}\ x\in\mathbb{R}.
	\end{equation}
	We notice that
	\begin{equation*}
		\|(v-1,u,\theta-1)(t_0')\|_3
		\leq C_6,\quad
		v(t_0',x)\geq C_2\quad {\rm for\ all}\ x\in \mathbb{R}.
	\end{equation*}
	Now we apply Lemma \ref{local}  again by
	taking $(v(t_0',\cdot), u(t_0',\cdot), \theta(t_0',\cdot))$  as the initial data.
	Then we derive that the solution $(v, u, \theta)$
	exists on $[t_0',t_0'+t_1']$ with
	$t_1'= \min\{T_1,T_0( C_2 ^{-1},
	\frac{1}{2},C_6)\}$,
	and for all $(t,x)\in[t_0',t_0'+t_1']\times\mathbb{R}$,
	\begin{equation*}
		\|(v-1,u ,\theta-1)(t)\|_3^2
		+\int_{t_0'}^t
		\left[\|v_{x}(s)\|_2^2+\|(u_{x},\theta_{x})(s)\|_3^2\right]\mathrm{d}s
		\leq 4C_6^2
	\end{equation*}
	and
	\begin{equation*}
		v(t,x)\geq \tfrac{1}{2}C_2,\quad
		\theta(t,x)\geq \tfrac{1}{4}.
	\end{equation*}
	Therefore, the solution $(v, u, \theta)$ satisfies \eqref{N3.6}
	for all $t\in[0,t_0'+t_1']$.
	
	We take $|\alpha|\leq \min\{\alpha_1,\alpha_2,\alpha_3\}$ with
	$\alpha_i$ $(i=1,2,3)$ being positive constants satisfying
	\eqref{alpha1}, \eqref{alpha2} and
	\begin{equation}
		\label{alpha3}
		\left(\tfrac{1}{4}\right)^{-\alpha_3}\leq 2,\quad
		\left(2\sqrt{5}C_6\right)^{\alpha_3}\leq 2,\quad
		\Xi\left(\tfrac{1}{2}C_2,
		\tfrac{1}4,\sqrt{5}C_6\right)\alpha_3\leq\epsilon_1,
	\end{equation}
	where the value of $\epsilon_1$ is chosen in Lemma \ref{L_v1}.
	Then we can deduce from
	Lemmas \ref{L_th2} and \ref{L_lower}  with $T=t_0'+t_1'$
	that for each   $t\in[t_0',t_0'+t_1']$,
	the local solution $(v(t,x), u(t,x), \theta(t,x))$ satisfies
	\eqref{N3.3}--\eqref{N3.3a} and
	\begin{equation} \label{N3.7}
		\theta(t,x)\geq \frac{\inf_{x\in\mathbb{R}}\theta(t_0',x)}
		{C_4 \inf_{x\in\mathbb{R}}\theta(t_0',x)T_1+1}
		\geq \frac{1}
		{C_4 T_1+2}=:C_7 \quad {\rm for\ all}\ x\in \mathbb{R} .
	\end{equation}
	Here we have used the estimate \eqref{3theta1}.
	Combining \eqref{N3.2} and \eqref{N3.7} yields that for each
	$t\in[0,t_0'+t_1']$,
	\begin{equation} \label{3theta2}
		\theta(t,x)
		\geq \min\left\{C_5,C_7\right\}
		:=C_8 \quad {\rm for\ all}\ x\in \mathbb{R} .
	\end{equation}
	We deduce from \eqref{3theta2} and
	Corollary \ref{cor1} that
	there exists some positive constant $C_9$,
	depending on
	$C_8$, $\Pi_0$,
	$V_0$, and
	$H(C_2)$,
	such that
	for each $t\in[0,t_0'+t_1']$,
	\begin{equation} \label{N3.8}
		\|(v-1,u ,\theta-1)(t)\|_3^2
		+\int_0^t
		\left[\|v_{x}(s)\|_2^2+\|(u_{x},\theta_{x})(s)\|_3^2\right]\mathrm{d}s
		\leq C_9^2.
	\end{equation}
	
	\vspace*{2mm}
	
	\noindent{\bf Step\,5}.
	Next if we take $(v(t_0'+t_1',\cdot), u(t_0'+t_1',\cdot), \theta(t_0'+t_1',\cdot))$ as the initial data,
	we  apply Lemma \ref{local} to construct the solution  $(v, u, \theta)$
	existing on the time
	interval $[0,t_0'+t_1'+t_2']$ with
	$$t_2'=\min\left\{T_1-t'_1,T_0(C_2,C_8,C_9)\right\},$$
	such that for all $(t,x)\in[t_0'+t_1',t_0'+t_1'+t_2']\times\mathbb{R}$,
	\begin{equation*}
		v(t,x)\geq \tfrac{1}{2}C_2,\quad
		\theta(t,x)\geq \tfrac{1}{2}C_8,
	\end{equation*}
	and
	\begin{equation}\label{N3.9}
		\|(v-1,u ,\theta-1)(t)\|_3^2
		+\int_{t_0'+t_1'}^t
		\left[\|v_{x}(s)\|_2^2+\|(u_{x},\theta_{x})(s)\|_3^2\right]\mathrm{d}s
		\leq 4C_9^2.
	\end{equation}
	Combine \eqref{N3.8} and \eqref{N3.9} to obtain that
	for all $t\in[0,t_0'+t_1'+t_2']$,
	\begin{equation}\label{N3.10}
		\|(v-1,u ,\theta-1)(t)\|_3^2
		+\int_{0}^t
		\left[\|v_{x}(s)\|_2^2+\|(u_{x},\theta_{x})(s)\|_3^2\right]\mathrm{d}s
		\leq 5C_9^2.
	\end{equation}
	
	Take $0<\alpha\leq \min\{\alpha_1,\alpha_2,\alpha_3,\alpha_4\}$,
	where $\alpha_i$ $(i=1,2,3)$ are positive constants satisfying
	\eqref{alpha1}, \eqref{alpha2},  \eqref{alpha3}, and
	\begin{equation}
		\label{alpha4}
		\left(\tfrac{1}{2}C_8\right)^{-\alpha_4}\leq 2,\quad
		\left(2\sqrt{5}C_9\right)^{\alpha_4}\leq 2,\quad
		\Xi\left(\tfrac{1}{2}C_2,
		\tfrac{1}{2}C_8,\sqrt{5}C_9\right)\alpha_4\leq \epsilon_1,
	\end{equation}
	where the value of $\epsilon_1$ is chosen in Lemma \ref{L_v1}.
	Then we infer from
	Lemma \ref{L_th2},
	Lemma \ref{L_lower} and Corollary \ref{cor1} with $T=t_0'+t_1'+t_2'$ that
	the local solution $(v(t,x), u(t,x), \theta(t,x))$ satisfies  \eqref{3theta2}--\eqref{N3.8}
	for each $t\in[0,t_0'+t_1'+t_2']$.
	By assuming $|\alpha|\leq \min\{\alpha_1,\alpha_2,\alpha_3,\alpha_4\}$,
	we can repeatedly apply the argument above to extend the local solution
	to the time interval $[0,t_0'+T_1]$.
	Furthermore, we deduce that
	\eqref{3theta2}--\eqref{N3.8} hold for each $t\in[0,t_0'+T_1]$.
	In view of $t_0'+T_1\geq {3}T_1/{2}$,
	we have shown that the Cauchy problem  \eqref{NS_L}--\eqref{far},
	\eqref{transport}--\eqref{h}
	admits a unique solution $(v, u, \theta)
	\in X(0,\frac{3}{2}T_1;C_2,C_8,C_9)$
	on the time interval $[0,\frac{3}{2}T_1]$.
	
	\vspace*{2mm}
	
	\noindent{\bf Step\,6}.
	We take $|\alpha|\leq \min\{\alpha_1,\alpha_2,\alpha_3,\alpha_4\}$.
	As in Steps 4 and 5,
	we can find $t_0''\in[t_0'+T_1/2,t_0'+T_1]$ such that
	the Cauchy problem  \eqref{NS_L}--\eqref{far},
	\eqref{transport}--\eqref{h} admits a unique solution
	$(v, u, \theta)$  on  $[0,t_0''+T_1]$, which satisfies
	\eqref{3theta2}--\eqref{N3.8} for each $t\in[0,t_0''+T_1]$.
	Since $t_0''+T_1\geq t_0'+{3}T_1/{2}\geq 2T_1$,
	we have extended the local solution $(v, u, \theta)$ to
	the time interval $[0,2T_1]$.
	Repeating the above procedure, we can then extend the
	solution $(v, u, \theta)$ step by step to a global one provided that
	$|\alpha|\leq \min\{\alpha_1,\alpha_2,\alpha_3,\alpha_4\}$.
	
	Choosing
	\begin{equation}\label{epsilon0}
		\epsilon_0=\min\{\alpha_1,\alpha_2,\alpha_3,\alpha_4\},
	\end{equation}
	where $\alpha_i$ $(i=1,2,3,4)$ are given by
	\eqref{alpha1}, \eqref{alpha2},
	\eqref{alpha3}, and \eqref{alpha4},
	we then derive that the Cauchy problem
	\eqref{NS_L}--\eqref{far},
	\eqref{transport}--\eqref{h}
	has a unique solution $(v,u,\theta)$  satisfying \eqref{N3.3}, \eqref{3theta2},
	and \eqref{N3.8} for each $t\in[0,\infty)$.
	Thus we have
	\begin{equation} \label{3_end}
		\sup_{0\leq t<\infty}\|(v-1,u,\theta-1)(t)\|_3^2
		+\int_0^{\infty}
		\left[\|v_x(t)\|_2^2+\|(u_{x},\theta_{x})(t)\|_3^2\right]\mathrm{d}t
		\leq C_9^2,
	\end{equation}
	from which we derive that
	the solution $(v, u,\theta)\in X(0,\infty;
	C_2,C_8,C_9)$.
	
	The large-time behavior \eqref{thm4} follows
	from \eqref{3_end}
	by using a standard argument
	(see  \cite{MN04book}).
	
	Recall that $\epsilon_1$, $C_i$ $(i=1,2,3,4,5,7,8)$
	depend only on
	$\Pi_0$,
	$V_0$, and
	$H(V_0)$,
	while $C_6$ and $C_9$ depend only on
	$\Pi_0$,
	$V_0$, and
	$H(C_2)$.
	According to the definition
	\eqref{epsilon0} of $\epsilon_0$,
	we can conclude the proof of Theorem \ref{thm}.

\bigbreak

\begin{center}
\textbf{Acknowledgement}
\end{center}

	The authors thank the anonymous referees for many helpful and 
	valuable comments that substantially improved the presentation of the paper.  


\bibliographystyle{siam}

\bibliography{bib_viscosity}

\end{document}